\documentclass[11pt,reqno]{amsart}

\usepackage{amsmath}
\usepackage{amssymb}
\usepackage{amscd}
\usepackage{amsfonts}
\usepackage{mathrsfs}
\usepackage{setspace}
\usepackage{version}
\usepackage{mathtools}

\usepackage{xcolor}

\definecolor{crimson}{rgb}{0.86, 0.08, 0.24}
\definecolor{bleudefrance}{rgb}{0.19, 0.5, 0.91}
\usepackage[pdftex,colorlinks,linkcolor=crimson,citecolor=bleudefrance,filecolor=black]{hyperref}
\usepackage[utf8]{inputenc}

\theoremstyle{plain}
\numberwithin{equation}{section} 
\newtheorem{theorem}{Theorem}[section]
\newtheorem{proposition}[theorem]{Proposition}
\newtheorem{lemma}[theorem]{Lemma}

\newtheorem{corollary}[theorem]{Corollary}

\newtheorem{claim}[theorem]{Claim}
\theoremstyle{remark}
\newtheorem{remark}[theorem]{Remark}

\usepackage{soul}

\renewcommand{\leq}{\leqslant}
\renewcommand{\geq}{\geqslant}
\newsavebox{\proofbox}
\savebox{\proofbox}{\begin{picture}(7,7)  \put(0,0){\framebox(7,7){}}\end{picture}}

\newcommand\E{\mathbb{E}}
\newcommand\Z{\mathbb{Z}}
\newcommand\R{\mathbb{R}}

\newcommand\p{\mathbb{P}}
\newcommand\C{\mathbb{C}}
\newcommand\N{\mathbb{N}}

\newcommand\kk{\mathrm{k}}

\newcommand\M{\operatorname{M}}

\newcommand\SL{\operatorname{SL}}

\newcommand\inte{\operatorname{int}}

\newcommand\GL{\operatorname{GL}}

\newcommand\Mat{\operatorname{Mat}}

\newcommand\diam{\operatorname{diam}}

\newcommand\Isom{\operatorname{Isom}}
\newcommand\Lip{\operatorname{Lip}}

\newcommand\lazy{\operatorname{lazy}}

\newcommand\supp{\operatorname{supp}}

\newcommand\Q{\mathbb{Q}}

\newcommand\id{\operatorname{id}}

\allowdisplaybreaks

\newcommand{\efface}[1]{}

\usepackage{datetime}

\begin{document}
 
\title[Concentration inequalities for random walks on hyperbolic spaces]{Random walks on hyperbolic spaces: Concentration inequalities and  probabilistic Tits alternative}

\author{Richard Aoun}
\address{American University of Beirut, Department of Mathematics, Faculty of Arts and Sciences, P.O. Box 11-0236 Riad El Solh, Beirut 1107 2020, Lebanon (on leave in New York University Abu Dhabi, PO Box 129188, Saadiyat Island, Abu Dhabi, United Arab Emirates)}
\email{ra279@aub.edu.lb}
\thanks{The first author is  supported  by a Research Group Linkage Programme from the Humboldt Foundation}

\author{Cagri Sert}
\address{Institut f\"{u}r Mathematik, Universit\"{a}t Z\"{u}rich, 190, Winterthurerstrasse, 8057 Z\"{u}rich, Switzerland}
\email{cagri.sert@math.uzh.ch}
\thanks{The second author is supported by SNF Ambizione grant 193481}


\maketitle

\begin{abstract} 
The goal of this article is two-fold: in a first part, we prove Azuma--Hoeffding type concentration inequalities around the drift for the displacement of non-elementary random walks on hyperbolic spaces. For a proper hyperbolic space $M$, we obtain explicit bounds that depend only on $M$, the size of support of the measure as in the classical case of sums of independent random variables, and on the norm of the driving probability measure in the left regular representation of the group of isometries. We obtain uniform bounds in the case of hyperbolic groups and effective bounds for simple linear groups of rank-one. In a second part, using our concentration inequalities, we give quantitative finite-time estimates on the probability that two independent random walks on the isometry group of a hyperbolic space generate a free non-abelian subgroup. Our concentration results follow from a more general, but less explicit statement that we prove for cocycles which satisfy a certain cohomological equation. For example, this also allows us to obtain subgaussian concentration bounds around the top Lyapunov exponent of random matrix products in arbitrary dimension. 
\end{abstract}

\setcounter{tocdepth}{1}
 
{
  \hypersetup{linkcolor=black}
  \tableofcontents
}

\section{Introduction}

Let $(M,d)$ be a metric space and $\Isom(M)$ the group of isometries of $M$. Consider a finitely supported probability measure $\mu$ on $\Isom(M)$, let $(X_i)_{i \in \N}$ be a sequence of independent random variables with distribution $\mu$ and denote by $R_n$ the random variable given by the product $X_1 \ldots X_n$. Fix a basepoint $o \in M$ and consider the random walk  $R_no$ on $M$. A straightforward application of Kingman's subadditive ergodic theorem shows that there exists a constant $\ell(\mu) \geq 0$, called the drift of the random walk,  such that
\begin{equation}\label{eq.kingman}
\frac{1}{n} d(R_n o, o) \underset{n \to \infty}{\overset{a.s.}{\longrightarrow}} \ell(\mu).
\end{equation}
This can be seen as a generalization of the classical law of large numbers which corresponds to the case $M=\R$ and $\mu$ supported on the translations $\R<\Isom(\R)$. 


Understanding various aspects of the convergence \eqref{eq.kingman} (e.g.\ central limit theorem (CLT), large deviation principles (LDP), Azuma-Hoeffding-type concentration inequalities) in the aforementioned special case constitutes a fundamental part of classical probability theory. Various other cases have attracted considerable attention relatively more recently: starting in '60s with the work of Furstenberg, Kesten, Oseledets, Kaimanovich \cite{furstenberg.poisson,furstenberg-kesten,kaimanovich.oseledets, oseledets} for symmetric spaces of non-compact type and with Dynkin--Malyutov \cite{dynkin-malyutov}, Furstenberg \cite{Furstenberg.boundary}, Kaimanovich--Vershik \cite{kaimanovich-vershik} and others for  random walks on countable groups. More recently, for general metric spaces with an assumption of coarse negative-curvature (namely Gromov hyperbolicity), a number of analogues of the classical results were proven including CLT's \cite{BQ.CLT.hyperbolic, mathieu-sisto}, local limit theorems \cite{gouezel.local}, and closer to our considerations, LDP's and exponential decay results \cite{BMSS, gouezel.cts}. Our goal in this paper is to establish Hoeffding-type concentration inequalities in the general setting of random walks on hyperbolic spaces. To the best of our knowledge, this aspect of the classical theory is far less developed in our setting.

Concentration inequalities around the mean $\ell(\mu)$ have two distinctive features compared to asymptotic large deviations estimates: on the one hand, these are large deviation bounds for the fluctuations of the distance of the random walk that are valid \textit{uniformly over all times} as opposed to asymptotic estimates. On the other hand, the exponential decay rate is expressed \textit{as an explicit function of the normalized deviation distance $t$.}
As such, these inequalities have been useful in the classical case both from a pure mathematics and applied or computational perspectives. Accordingly, one of the main reasons that we mostly focus our attention in this article to proper Gromov hyperbolic spaces is that, by following a geometric and harmonic analytic technique of Benoist--Quint \cite{BQ.CLT.hyperbolic}, we are able to exploit their geometry and consequently obtain explicit concentration estimates. 
We also obtain subgaussian concentration estimates for non-proper Gromov hyperbolic spaces and random matrix products, but with less explicit bounds. These results are also new and discussed later in the introduction.

Our approach consists of proving a general concentration type result for cocycles satisfying a certain cohomological equation. This is line with Gordin's method for proving the central limit theorem where the values of cocycles along random walks coming from group actions are related to martingales via a Poisson type equation.

In particular, the  solutions by Benoist--Quint of associated cohomogical equations for Busemann and norm cocycles, respectively on the boundary of hyperbolic spaces \cite{BQ.CLT.hyperbolic} and projective spaces \cite{BQ.CLT.linear}, play a crucial role in the application of our general cocycle-concentration results to these  settings. We slightly extend this solution to adapt it to our purposes, and in the case of proper hyperbolic spaces, we get explicit bounds on its size. These bounds involve the norm $\|\lambda_G(\mu)\|_2$ of the regular representation $\lambda_G$ of a probability measure $\mu$ on the isometry group $G=\Isom(M)$. In a later part, we use various versions of uniform Tits alternatives to control the size of $\|\lambda_G(\mu)\|_2$ which in turn yields effective constants for example in the case of linear groups of rank one, thanks to the works of Breuillard \cite{breuillard.strong.tits, breuillard.height}. 

Finally, we give explicit finite-time estimates for the probability that two independent non-elementary random walks on a proper hyperbolic space generate a free subgroup. We deduce this result from our concentration bounds together with a more  general statement linking uniform large deviations with free-subgroups generated by samplings of random walks.
Our result (Theorem \ref{thm.proba.tits}) quantifies some cases of several known  probabilistic Tits alternatives proven in  \cite{aoun.tits,GMO,taylor.tiozzo}.

Let us now state our first main result, some of its consequences and related remarks.
\subsection{Subgaussian concentration estimates for random walks on hyperbolic spaces}
We first introduce some notation and definitions.

Let $(M,d)$ be a proper metric space, we denote by $G$ its group of isometries. It is a locally compact group and we denote by $\mu_G$ a Haar measure on $G$. For every $r \in [0,1]$, we denote $\mu_{r,\lazy}=r \delta_{\id} + (1-r)\mu$. Furthermore, we denote by $\lambda_G(\mu)$ the operator given by the image of the probability measure $\mu$ under the the left-regular representation of $G$ on $L^2(G)$. 
Finally, having fixed a basepoint $o \in M$, for an element $g \in G$, we set $\kappa(g):=d(go,o)$ and for a set $S \subset G$, $\kappa_S:=\sup\{\kappa(g) : g\in S\}$. The set $S$ is said to be bounded if $\kappa_S<\infty$.

Given $\delta \geq 0$, by a $\delta$-hyperbolic metric space $M$, we understand a metric space $M$ such that for every $x,y,z,o \in M$, we have $(x|y)_o \geq (x|z)_o \wedge (z|y)_o -\delta$, where $(.|.)_.$ is the Gromov product given by $(x|y)_o=\frac{1}{2}(d(x,o)+d(y,o)-d(x,y))$. A probability measure $\mu$ is called \textit{non-elementary} if its support $S$ generates a semigroup that contains two independent loxodromic elements (see \S \ref{subsec.hyp}). We can now state

\begin{theorem}\label{thm.main.with.lambda}
Let $(M,d)$ be a proper geodesic $\delta$-hyperbolic space and $o\in M$. Assume that the group $G=\Isom(M)$ acts cocompactly on $M$. Then, there exists an explicit positive function $D(.,.)$ with $D(.,\lambda)<\infty$ for every $\lambda \in (0,1)$ such that for every  non-elementary probability measure $\mu$ on $G$ with bounded support $S$, for every $t \geq 0$ and $n \in \mathbb{N}$ we have  
\begin{equation}\label{eq.main}
\p\left( |\kappa(R_n) - n \ell(\mu)|\geq nt \right) \leq 2
\exp \left(\frac{-nt^2}{ \kappa_S^2 D(\kappa_S,\|\lambda_G(\mu_{r,\lazy})\|_2)} \right)
\end{equation}
for every $r \in [0,1)$.
\end{theorem}

This statement will follow from a more general concentration result (Theorem \ref{thm.main.text}) for the Busemann cocycle on the horofunction compactification of $M$.

To convey the dependence of this upper bound to the involved quantities and for practical use, in the following remark we provide a function that one can substitute for the function $D$ in the previous result.  

\begin{remark}[On the upper bound]\label{rk.D}
One can take  
$$D(\kappa, \lambda)= 32 \left(
   16\ln^+ (\kappa)+8A_0/3+33 \right)^2 \frac{1}{(1-\sqrt{\lambda})^4},$$
where $A_0= \left(\frac{\mu_G \left(B_{2R(\delta)+2D_0} \right)}{\mu_G \left(B_{R(\delta) +D_0}\right)} \right)^{1/2}$ with $R(\delta)=14\delta+4$, for $r \geq 0$, $B_r:=\{g \in G : d(go,o) \leq r\}$, and $D_0:=2\textrm{diam}( G\backslash M)$. We also set $D(\kappa,1)=\infty$. Note that if $\mu$ is non-elementary, then for every $r \in (0,1)$, we have $\|\lambda_G(\mu_{r,\lazy})\|_2<1$ (see Remark \ref{rk.lazy.discussion}). Also note that if $\mu$ is symmetric, then $\|\lambda_G(\mu)\|_2=\|\lambda_G(\mu_{0,\lazy})\|_2<1$. 
\end{remark}

\begin{remark}\label{rk.as.mentioned}
1. (Non-proper case)
As mentioned earlier, we also obtain subgaussian concentration estimates without the properness assumption but in this case, the dependence on $\mu$ at the right-hand-side of \eqref{eq.main} is less explicit (Proposition \ref{prop.hyp}).\\[2pt]
2. (Random walks with unbounded support) It is possible to have a version of our result where the bounded support assumption on the probability measure $\mu$ is replaced by a finite exponential moment assumption and obtain a Bennett--Bernstein type concentration inequality. However, the constants that appear in that version are more complicated to express. This point is discussed in more detail in Remark \ref{rk.unbounded.support}.
\end{remark}

In the sequel, we will see that each of the two aspects of the upper bound in Theorem \ref{thm.main.with.lambda}, namely its subgaussian form and its parameters of dependence, have implications and strenghtenings. On the one hand,   by combining this  upper bound with versions of uniform Tits alternatives in various contexts (which entail uniform bounds for $\|\lambda_G(\mu)\|_2$, see Lemma \ref{lemma.uniformtits.radius}), we will obtain uniform concentration estimates for a class of driving probability measures, see Corollaries \ref{corol.Koubi} and \ref{corol.manu}. On the other hand, the subgaussian character allows us for instance to provide a global quadratic lower bound (see Corollary \ref{corol.LDP}) for the rate function of large deviations, recently studied in this setting by \cite{BMSS}. Let us now explain these consequences.

\vspace*{0.1cm}

\subsubsection{The case of hyperbolic and rank-one linear groups} Firstly, specifying Theorem \ref{thm.main.with.lambda} to hyperbolic groups, and using Koubi's uniform Tits alternative \cite[Theorem 5.1]{koubi}, we obtain the following more precise concentration result for random walks on hyperbolic spaces. 

\begin{corollary}\label{corol.Koubi}
Let $(M,d)$ be a proper geodesic hyperbolic metric space and $o\in M$. Then there exists a constant $A_M>0$ such that for any group $\Gamma<G$ that acts properly and cocompactly on $M$, there exist constants $\alpha_\Gamma>0$ and $N_\Gamma \in \mathbb{N}$ depending only on $\Gamma$ such that for every non-elementary probability measure $\mu$ of finite support $S$ generating $\Gamma$, for every $t>0$ and $n\in \N$, setting $m_\mu=\min_{g \in S}\mu(g)$, we have
$$\p\left( |\kappa(R_n) - n\ell(\mu)|\geq nt \right) \leq
2 \exp\left(\frac{-nt^2}{m_\mu^{N_\Gamma} \alpha_\Gamma \kappa_S^2 (\ln^+(\kappa_S) + A_M)}\right)$$
\end{corollary}



\begin{remark}\label{rk.caval-sambu}
Using the quantitative Tits alternative in the recent work of Cavallucci--Sambusetti \cite[Theorem 1.1]{cavallucci-sambusetti}, under additional assumptions on the hyperbolic space $(M,d)$ (such as the existence of a convex geodesic bicombing with certain properties) and for a torsion-free group $\Gamma$, one can provide a version of the previous corollary dropping the cocompactness assumption of the $\Gamma$-action and replacing the constants $\alpha_\Gamma$ and $N_\Gamma$ with constants depending only on a packing parameter of the hyperbolic space $M$ (see \cite[\S 2.2]{cavallucci-sambusetti}).
\end{remark}

Specifying Theorem \ref{thm.main.with.lambda} to rank one matrix groups and using the strong Tits alternative of Breuillard \cite{breuillard.strong.tits, breuillard.height}, we obtain concentrations for 
random matrix products of discrete 
non-amenable subgroups of rank-one semisimple
algebraic groups. A further aspect of the following corollary is that thanks to the work of Breuillard, the implied constants can be effectively calculated.

We need some notation to state the next corollary. Let $\kk$ be a local field (i.e.\ in characteristic zero $\R$, $\C$ or a finite extension of $\Q_p$ for a prime number $p$ and in positive characteristic, a finite extension of $\mathbb{F}_p((T))$). We denote by $\|\cdot\|$ the canonical norm on $\kk^d$ for a fixed discrete valuation on $\kk$ and consider 
the associated operator norm on the space of $d\times d$-matrices. Moreover, if $S$ is a finite subset of $\Mat_d(\kk)$, we denote by $\kappa_S:=\sup\{\ln \|g\| : g\in S\}$. 
Finally, if $\mu$ is a probability measure with finite first order moment on $\textrm{GL}_d(\kk)$, we denote by 
$\ell(\mu)$ the top Lyapunov exponent, i.e.~ the almost sure limit of $ \frac{1}{n} \ln \|R_n\|$.

\begin{corollary}\label{corol.manu} 
Let $\kk$ be a local field and $\mathbb{H}\subseteq \mathrm{SL}_d$ be a connected semisimple linear algebraic group of $\kk$ rank-one defined over $\kk$. 
For every $d\in \N$, there exist  constants $\alpha_d>0$, $N_d \in \N$ depending only on the dimension $d$ and constants $A=A(\mathbb{H}, \kk)$ such that  for every finitely supported probability measure $\mu$ whose support generates a non-amenable discrete subgroup of $\mathbb{H}(\mathrm{k})$, for every $t>0$ and $n \in \N$, the following holds: 
\begin{equation}\p \left( | \frac{1}{n} \ln \|R_n\| - \ell(\mu)|   \geq t \right)\leq 2 \exp\left(\frac{-  n  t^2}{m_\mu^{N_d} \alpha_d \kappa_S^2(\ln^+(\kappa_S)+A_{\mathbb{H}, \kk})^2} \right).\label{concentration2}\end{equation}
\end{corollary}

\begin{remark}[About the discreteness assumption]
1. Both of the above corollaries are obtained from Theorem \ref{thm.main.with.lambda} in the following way: the respective versions of Tits alternatives allow us to deduce bounds on the norm $\|\lambda_\Gamma(\mu)\|$ of the regular representation on $\ell^2(\Gamma)$, which is equal to $\|\lambda_G(\mu)\|$ thanks to the discreteness assumption. In general, even though we have uniform upper bounds for $\|\lambda_\Gamma(\mu)\|$, we are not able to transfer this to a bound on $\|\lambda_G(\mu)\|$ without discreteness assumption. Indeed, by \cite{manu2003,kuranishi}, in any connected semisimple Lie group $G$, for any element $g \in G$, one can find pairs of elements $\{a_n,b_n\}$ that converge to $g$ and that generate a non-abelian free group, so that for the uniform probability measure $\mu_n$ supported on $\{a_n,b_n,a_n^{-1},b_n^{-1}\}$, we have $\frac{\sqrt{3}}{2}=\|\lambda_\Gamma(\mu_n)\|<\|\lambda_G(\mu_n)\| \to 1$.\\[2pt]
2. We also note that under the discreteness assumption, the fact that the support $S$ generates a non-elementary group implies, thanks to various versions of Margulis Lemma, a positive lower bound for $\kappa_S$. This lower bound depends in Corollary \ref{corol.Koubi} on some parameters of $M$ and the group generated by $S$ (see \cite[Theorem 5.21]{besson-courtois-gallot-sambusetti}). In Corollary \ref{corol.manu}, it depends only on $\mathbb{H}(\kk)$ (see e.g.\ \cite[Chapter 8]{ballmann-gromov-schroeder}). 
\end{remark} 

\vspace*{0.1cm}

\subsubsection{Rate function of LDP}\label{subsec.ldp}

We now mention a consequence of Theorem \ref{thm.main.with.lambda} concerning the rate function of large deviation principles of random walks on hyperbolic spaces recently studied by \cite{BMSS}. The authors prove that the sequence of random variables
$\frac{\kappa(R_n)}{n}$ satisfies a \emph{large deviation principle}  with proper convex rate function $I_\mu: [0,\infty) \to [0,+\infty]$ vanishing only at the drift $\ell(\mu)$. Recall that this means that $I_\mu$ is a lower-semicontinuous function such that for every measurable subset $J$ of $\R$, we have 
\begin{equation}\label{eq.def.LDP}
\underset{\alpha \in \inte(J)}{-\inf I(\alpha)} \leq \underset{n \rightarrow \infty}{\liminf} \frac{1}{n}\ln \mathbb{P}(\frac{\kappa(R_n)}{n} \in J) \leq \underset{n \rightarrow \infty}{\limsup} \frac{1}{n}\ln \mathbb{P}(\frac{\kappa(R_n)}{n} \in J) \leq \underset{\alpha \in \overline{J}}{-\inf I(\alpha)}
\end{equation}
where $\inte(J)$ denotes the interior and $\overline{J}$ the closure of $J$.
To the best of our knowledge, no explicit global estimate for the rate function exists in the literature. Theorem \ref{thm.main.with.lambda} allows us to give an explicit quadratic lower bound for the rate function $I_\mu$ in our setting, i.e.\ when $M$ is proper and the non-elementary probability measure $\mu$ has a bounded support. 
 
\begin{corollary}[Quadratic lower bound]\label{corol.LDP}
Under the assumptions of Theorem \ref{thm.main.with.lambda}, for every $t \in [0,\infty)$ the rate function $I_\mu$ of the sequence $\frac{1}{n}\kappa(R_n)$ satisfies
$$I_\mu(t)\geq \frac{(t-\ell(\mu))^2}{\kappa_S^2 D\left(\kappa_S, \|\lambda_G(\mu_r, \lazy)\|_2 \right)},$$
for every $r \in [0,1)$.
\end{corollary}

The proof of this corollary is immediate from the property \eqref{eq.def.LDP} defining the function $I_\mu$ and the estimate given by Theorem \ref{thm.main.with.lambda}.

\begin{remark}
In the general case of random walks with finite exponential moment, one can clearly not get such a quadratic lower bound, see Remark \ref{rk.unbounded.support} for the type of global lower bound that one can obtain using our methods.
\end{remark}

\subsection{Quantitative probabilistic Tits alternative}

It is known since the foundational work of Gromov \cite{gromov} that groups acting non-elementarily on hyperbolic spaces contain non-abelian free subgroups. The main result of this part is a probabilistic quantification of this fact which says that if we sample two independent random walks at their $n^{th}$-steps, the probability that the two elements generate a free group of rank two is exponentially close to one. Moreover, an important aspect is that this probability is explicitly described in terms of the norm of the driving measure $\mu$ in the regular representation and the size of its support.

\subsubsection{Probabilistic free-subgroup theorem}

\begin{theorem}\label{thm.proba.tits}
Keep the assumptions of Theorem
\ref{thm.main.with.lambda}. Then, there exist explicit functions $n_0(\cdot)$ and  $T(\cdot, \cdot)$ both with values in $(0,+\infty)$ such that for any non-elementary probability measure $\mu$ on $G=\Isom(M)$, denoting $(R_n)_{n\in \N}$ and $(R'_n)_{n\in \N}$ two independent random walks driven by $\mu$, for every 
$n> n_0 \left (\|\lambda_{G}(\mu_{1/2, \lazy})\|_2 \right)$, we have 
\begin{equation}\label{eq.proba.tits}
\p \left( \langle R_n, R'_n \rangle \,\textrm{is free} \right) \geq 1-50 \exp\left(-n\, T(\kappa_S, \|\lambda_{G}(\mu_{1/2, \lazy})\|_2) \right).
\end{equation}
\end{theorem}

We proceed with a few remarks on the statement and some consequences.

\begin{remark}[The explicit estimate]\label{rk.explicit.tits}${}$\\
(i) For the function appearing in the above statement, one can take
$$T(\kappa, \lambda)=  \frac{1}{A_M} \frac{(\ln\lambda)^2 (1-\sqrt{\lambda})^4}{\kappa^2(\ln^+(\kappa) +1)^2} \qquad \text{and} \qquad n_0(\lambda)=2-A_M  \frac{1}{\ln \lambda},$$
where the constant $A_M>0$ is related only to a doubling constant of the Haar measure on $G$ and to the diameter of $G \backslash M$ 
 (see  \eqref{eq.explicit.guy} for its expression).
\\[2pt]
(ii)  Unlike in our previous results, the left-hand-side \eqref{eq.proba.tits} is independent of the choice of basepoint $o$. One can therefore replace $\kappa_S$ by the joint minimal displacement $L(S)$ of $S$  (\cite{breuillard-fujiwara}) given by $\inf_{x \in M} \sup_{s \in S} d(sx,x)$, which is independent of any basepoint.\\[2pt]
(iii) Finally, the choice of $1/2$ for the lazy random walk $\mu_{1/2,\lazy}$ is for convenience: it ensures that the associated operator norm is strictly less than one (which might not be the case for $\mu$ due to the non-symmetry of $\mu$, see Remarks \ref{rk.lazy.discussion} and \ref{rk.lazy.again}).
\end{remark}

\begin{remark}
Using similar techniques, one can also prove a more general version of this result where several (more than two) independent copies of random walks, even with different step-distributions, are considered.
\end{remark}

\subsubsection{Some consequences}\label{subsub.tits.conseq}
\textbullet ${}$ For discrete subgroups of $\Isom(M)$, in the respective settings, using Corollaries \ref{corol.Koubi} and \ref{corol.manu} (see also Remark \ref{rk.caval-sambu}), we can deduce an explicit expression for the right-hand-side of \eqref{eq.proba.tits} as well for its range of validity controlled by $n_0(\cdot)$  (see Remark \ref{rk.last.rk}).\\[2pt]
\textbullet ${}$ Moreover, it is known that for a discrete subgroup of isometries $\Gamma$ of a proper geodesic hyperbolic space $M$ such that $\Isom(M)$ acts cocompactly on $M$, the group $\Gamma$ is either virtually nilpotent or non-elementary (see e.g.\ \cite[Corollary 3.13]{cavallucci-sambusetti}). Hence Theorem \ref{thm.proba.tits} can be seen as a quantitative probabilistic Tits alternative for discrete groups of isometries of $M$.\\[2pt]
\textbullet ${}$ Theorem \ref{thm.proba.tits} gives an explicit version of a result by Taylor--Tiozzo \cite[Corollary 1.6]{taylor.tiozzo} under additional hypotheses. We also refer to Gilman--Miasnikov--Osin \cite[Theorem 1.2]{GMO} for a previous result in the particular case of Gromov-hyperbolic groups.
Finally, in the setting of discrete subgroups of rank one semi-simple linear algebraic groups, Theorem \ref{thm.proba.tits} provides an effective version of the probabilistic Tits alternative proved by the first-named-author (see \cite[Theorem 1.1]{aoun.tits}).

\subsection{Random matrix products}
The concentration estimates that we obtain in Section \ref{sec.cocycle} for general cocycles also allow us to deduce concentration estimates for random matrix products in arbitrary dimension, but these are less explicit compared to Theorem \ref{thm.main.with.lambda}. Before stating the result we recall some known facts; we refer to \S \ref{subsec.matrix} for more details.  
Let $\mu$ be a probability measure on $\GL_d(\C)$ whose support generates a strongly irreducible and proximal subgroup, then there exists a unique $\mu$-stationary probability measure $\nu$ on the projective space of $\C^d$ (\cite{Furstenberg.boundary,guivarch-raugi}). The stationary measure $\nu$ enjoys some regularity properties. It is non-degenerate 
(i.e.\ does not charge any proper hyperplane) \cite{Furstenberg.boundary}, log-regular under a finite second order moment \cite{BQ.CLT.linear} and H\"{o}lder regular under a finite exponential moment assumption \cite{Guivarch3}. Suppose now $\mu$ has bounded support and consider $\mathfrak{c}(\mu):=
\sup_{x\in \C^d\setminus \{0\}} \int \ln \frac{\|x\|\,\|y\|}{|\langle x, y\rangle|} d\nu(\C y)$. It follows from the aforementioned regularity properties that this quantity is finite. Finally, we denote $\mu^\ast$ the pushforward of $\mu$ by the map $g\mapsto g^{*}$, where $g^*$ is the conjugate-transpose of $g$. With these at hand, we are now ready to state

\begin{proposition}
Let $\mu$ be a boundedly supported probability measure on $\GL_d(\C)$ such that the semigroup generated by the support $S$ of $\mu$ is strongly irreducible and proximal.
Let $\kappa_S:=\max\{\ln\|g\|  \vee  \ln \|g^{-1}\|; g \in S\}$ and  $\mathfrak{c}=\mathfrak{c}(\mu^\ast)$. Then, for every $t>0$ and $n \in \mathbb{N}$, we have
$$ \sup_{v\in \C^d\setminus \{0\}}\p\left( \left| \ln \frac{\|R_n v \|}{   \|v\|} - n\ell(\mu)\right|\geq nt \right) \leq 2\exp\left(-\frac{nt^2}{32  (\kappa_S +{\mathfrak{c)}^2}}\right).$$
 In particular, 
 for every $t>0$ and  $n \in \mathbb{N}$ such that $nt \geq \ln d$, the following holds: 
$$\p\left(  \left| \ln \|R_n \|    - n\ell(\mu)\right|\geq nt \right) \leq 2d
\exp\left(-\frac{nt^2}{128  (\kappa_S+ \mathfrak{c})^2}\right).$$
\label{hoeffding-high-dimension}
\end{proposition}
In this result, the fact that we have subgaussian estimates for every $t>0$ small enough can also be deduced from the spectral gap result of Le Page \cite{page} using analytic perturbation methods. We also refer to \cite{BQ.CLT.linear,boyer} for exponential deviation estimates in a more general setting and to \cite[Ch.~5]{duarte-klein.book} for local concentrations that are uniform over small neighborhoods of irreducible cocycles.

\begin{remark}
Similarly to Corollary \ref{corol.LDP}, the estimate in Proposition \ref{hoeffding-high-dimension} allows one to obtain a global lower bound (less explicit in its constants compared to the aforementioned corollary) for the rate function of log-norms of random matrix products studied in \cite{sert.matrix.LDP, xiao-hui-grama.precise.ld}   (see also \cite[Corollary 4.17]{sert.matrix.LDP}). 
\end{remark}

We end the introduction by mentioning that\\
\textbullet ${}$ the methods we use to prove Theorem \ref{thm.main.with.lambda} allow us to provide an explicit lower bound for the bottom of the support of Hausdorff spectrum of the harmonic measure, equivalently, for the exponent with which the Frostman property holds (see \S \ref{subsec.log.track} and see also Tanaka \cite{tanaka.spectrum} for a thorough discussion of multifractal analysis of the harmonic measure in the particular case of hyperbolic groups);\\[2pt]
\textbullet ${}$ Theorem \ref{thm.main.with.lambda} itself has a direct application to the continuity of the drift (\S \ref{subsec.continuity});\\[2pt]
\textbullet ${}$  in view of Horbez's work \cite{horbez}, it seems possible that our results in \S \ref{sec.cocycle} can be used to obtain subgaussian concentration estimates in the setting of random walks on mapping class groups and on the group $Out(F_N)$ of outer automorphisms of a non-abelian free group.

\subsection*{Organization}
The article is organized as follows. In Section \ref{sec.cocycle}, we prove concentration estimates for a general cocycle that satisfies a certain cohomological equation (Proposition \ref{prop.cocycle}). In Section \ref{sec.non.explicit}, we deduce non-explicit concentration estimates for random matrix products in arbitrary dimension (Proposition \ref{hoeffding-high-dimension}) and for random walks on hyperbolic spaces (Proposition \ref{prop.hyp}). In Section \ref{sec.explicit}, we prove Theorem \ref{thm.main.with.lambda}. In Section \ref{sec.koubi.manu}, we prove Corollaries \ref{corol.Koubi} and \ref{corol.manu}. Finally in Section \ref{sec.tits}, we deduce Theorem \ref{thm.proba.tits} from Theorem \ref{thm.main.with.lambda}, a uniform positive lower bound on the drift (Proposition \ref{prop.drift.lower.bound}) and a general result estimating the likelihood of obtaining free subgroups from random walks based on uniform large deviation estimates (Proposition \ref{prop.ULD.tits}). 

\subsection*{Acknowledgements}
This article profited from helpful discussions with several people; the authors have the pleasure to thank R\'{e}mi Boutonnet, Emmanuel Breuillard, Yves Cornulier, Yves Guivarc'h, Andrea Sambusetti and Pierre Youssef. We also thank Samuel Taylor for helpful bibliographical suggestions. Finally, we are grateful to the anonymous referee for a careful reading of the article.

\section{Concentration inequalities for cocycles satisfying a Poisson equation}\label{sec.cocycle}
 
The goal of this section is to prove Proposition \ref{prop.cocycle} yielding concentration inequalities for values of a cocycle for which the associated Poisson equation has a bounded measurable solution. This result will provide the basis for the rest of the article where we will obtain more precise versions in the particular setups discussed in Introduction. We note that this section is inspired by the work of Furstenberg--Kifer \cite{furstenberg-kifer} of which it can be seen as a quantitative analogue under an additional assumption (see Remark \ref{rk.furstenberg-kifer}). 

We start by recalling some standard terminology. Let $G$ be a Polish group (endowed with the Borel $\sigma$-algebra) and $X$ a standard Borel space endowed with a measurable action of $G$. We shall refer to such a space as a $G$-space. A function $\sigma:G \times X \to \mathbb{R}$ is said to be an additive cocycle if it satisfies $\sigma(g_1g_2,x)=\sigma(g_1, g_2x)+ \sigma(g_2,x)$ for every $g_1,g_2 \in G$ and $x \in X$. All cocycles will supposed to be measurable. Given a probability measure $\mu$ on $G$, a probability measure $\nu$ on $X$ is said to be $\mu$-stationary if for every bounded measurable function $\phi$, we have $\int \int \phi(gx) d\mu(g) d\nu(x)=\int \phi(x) d\nu(x)$. We denote by $P_\mu$ the Markov operator acting on bounded measurable functions on $X$ by $P_\mu \phi(x)=\int \phi(gx) d\mu(g)$. Finally, denoting by $(X_i)_{i \in \N}$ a sequence of independent $G$-valued random variables with distribution $\mu$, we write $L_n$ for the left product $X_n\cdots X_1$. Although, $L_n$ and $R_n$ have the same distribution, it will be more convenient in this section to work with the left random walk $L_n$.


\begin{proposition}\label{prop.cocycle}
Let $G$ be a Polish group, $X$ a $G$-space and $\sigma: G\times X\to \R$ a bounded additive cocycle. Let $\mu$ be a probability measure on $G$ with support $S$. Denote by $$\kappa_S:=\sup\{\sup_{x\in X}|\sigma(g, x)| : g\in S\}.$$
Let $\nu$ be a $\mu$-stationary probability measure on $X$ and 
$$\ell(\mu):=\int_{G\times X}{\sigma(g,x) d\mu(g) d\nu(x)}.$$ 
Assume that the set $E$ of bounded measurable solutions $\psi$ of the Poisson equation 
\begin{equation}\psi(x) - P_{\mu}(\psi)(x)=\int_{G}{\sigma(g,x) d\mu(g)} - \ell(\mu).\label{cohomological}
\end{equation}
is non-empty and let $\mathfrak{c}:= \inf\{\|\psi\|_{\infty} : \psi\in E\}$. Then, for every $t>0$, $n\in \N$, and $x\in X$ we have
$$\p\left( |\sigma(L_n, x) - n\ell(\mu)|\geq nt \right) \leq 2\exp\left(-\frac{nt^2}{32  (\kappa_S +{\mathfrak{c)}^2}}\right).$$
\end{proposition}

\begin{remark}\label{rk.furstenberg-kifer} 1. Our assumption \eqref{cohomological} implies that there is a unique cocycle average in the sense of \cite[\S 3]{BQ.CLT.linear}.\\[1pt]
2. This result can be seen as an abstract quantitative refinement of \cite[Theorem 2.1]{furstenberg-kifer} under the assumption that the expected increase function is cohomologous to a constant.
\end{remark}

The proof of the previous result is based on the following general probabilistic ingredient. We start by recalling some standard terminology on Markov chains. Let $M$ be a standard Borel space, $P$ a Markov operator on $M$, i.e.\ a measurable map $x \mapsto P_x$ from $M$ to the space of probability measures on $M$. This data naturally defines an operator on the space of bounded Borel functions on $M$ by $\phi \mapsto P\phi$, where $P \phi(x)= \int \phi(y) dP_x(y)$. Given $x \in M$, we denote by $\mathbb{P}_x$ the law of the Markov chain $(Z_n)_n$ on the space of trajectories, i.e.\ $M^\mathbb{N}$ and $\mathbb{E}_x$ the associated expectation operator. We say that a probability measure $\pi$ is invariant (or stationary) under the Markov operator $P$ if $\int P\phi d\pi = \int \phi d\pi$ for every bounded measurable function $\phi$ on $M$.

\begin{proposition}\label{prop.Markov}
Let $(Z_n)$ be a Markov chain on a standard Borel space $M$ associated to the Markov operator $P$. Let $\pi$ be a $P$-stationary probability measure on $M$. Let $f$ be a bounded measurable function on $M$. We assume that $f$ is cohomologous to $\int_{M} {f\,d\pi}$, i.e.~ there exists a bounded measurable solution $\phi$ of the equation:
\begin{equation}\label{cohomological-MC} \phi  - P\phi=f - \int_{M} {f\,d\pi}.
\end{equation}
Then, for every $t>0$, $n\in \N$ and $x \in M$, the following inequality holds 
$$\p_{x} \left(     \Big| \sum_{i=1}^n{f(Z_i)} - n \int_{M} {f\,d\pi} \Big|  \geq nt \right)   \leq  2 \exp\left(-\frac{n t^2}{32 \|\phi\|_{\infty}^2}\right).$$
\end{proposition}

\begin{remark}
1. Let $M$ be a compact metric space and $f: M\to \R$ a continuous function. Suppose, for simplicity, that the operator $P$ is Markov-Feller and that $f$ has a unique average $\ell_f:=\int f d\pi$ for all $P$-stationary probability measures $\pi$. Even though $f$ may not be cohomologous to the constant $\ell_f$, Furstenberg--Kifer \cite[Lemma 3.1]{furstenberg-kifer} showed that for every $t>0$, there exists a continuous 
function $h_{t}$ on $M$, cohomologous to $f$ and such that $\|h_{t}\|_{\infty}\leq \ell_f + t$. This can be used to show the exponential decay of $\p(|\sum_{i=1}^n{f(Z_i) - n \ell_f|>nt})$. This is a particular case of Benoist--Quint's \cite[Proposition 3.1]{BQ.CLT.linear}. In Proposition \ref{prop.Markov}, thanks to the stronger assumption \eqref{cohomological-MC}, one obtains subgaussian exponential decay with explicit constants. \\[2pt]
2. In some particular cases, powerful concentration inequalities exist for the sums of \emph{any} function along the Markov chain \cite{ gouezel-dedecker,Gillman}. They are not applicable here since our Markov chains are not geometrically ergodic. On the other hand, the particular requirement \eqref{cohomological-MC} on the function $f$ allows us to use the usual Hoeffding inequality for martingales and thereby deduce the previous concentration estimates in the generality of Markov chains that we consider.
\end{remark}

\begin{proof}[Proof of Proposition \ref{prop.cocycle}]
We start by defining the appropriate objects to which we will apply Proposition \ref{prop.Markov}. We take the standard Borel space $M$ to be $S \times X$ and $P$ the Markov operator defined  by 
$$Pf\left((g,x)\right) = \int_{G}  { f(\gamma, \gamma x)\, d\mu(\gamma)}$$
for every bounded measurable function $f$ on $M$.
The associated Markov chain $(Z_n)_{n\in \N}$ on $M$ starting from $Z_0=(e,x)$ is the process
$$Z_0=(e,x)\,,\, Z_1=(g_1, g_1\cdot x)\,,\, Z_2=(g_2, g_2 g_1 \cdot x)\,\cdots \,  Z_n=(g_n, L_n\cdot x), \cdots ,$$
where the $g_i$'s are iid random variables on $G$ with distribution $\mu$.
Let $\pi$ be the probability measure on $M$ defined by 
$$\int_{M}\,{f\,d\pi} := \iint_{S\times X} {f(g,g\cdot x)\, d\mu(g)\,d\nu(x)}$$
for every bounded measurable $f$ on $M$. Since $\nu$ is a  $\mu$-stationary, one readily checks that $\pi$ is stationary   for the Markov operator $P$.  Let now  $$f: M \longrightarrow \R, (g,x) \longmapsto f(g,x):=\sigma(g,g^{-1} \cdot x).$$
The following properties are immediate to check
\begin{itemize}
\item  Starting from  $Z_0=(e,x)$, we have $\sum_{i=1}^n {f(Z_i)} = \sigma(L_n,x)$, 
\item $\int_{M} f\,d\pi = \iint f(g,g\cdot x)\,d\mu(g)\,d\nu(x) = \iint \sigma(g,x) \,d\mu(g)\,d\nu(x)=\ell(\mu)$
\item $\|f\|_{\infty} \leq  \kappa_S$.
\end{itemize}
Finally, we check that if \eqref{cohomological} holds for some $\psi$, then \eqref{cohomological-MC} holds. 
Indeed, let 
$$\phi: M \longrightarrow \R, (g,x) \longmapsto   \phi(g,x):=  \psi(x)+f(g,x). $$ 
One readily checks that $P\psi=P_{\mu}\psi$  and $Pf(g,x)= \int_G  \sigma(g,x)\,d\mu(g)$. Thus, by \eqref{cohomological},  $\phi - P \phi=f - \int_M{f \,d\pi}$, and \eqref{cohomological-MC} is fulfilled. Since $\|\phi\|_{\infty}\leq \|\psi\|_{\infty}+\kappa_S$, Proposition \ref{prop.cocycle} follows from Proposition \ref{prop.Markov}. 
\end{proof}

\begin{proof}[Proof of Proposition \ref{prop.Markov}]
Let $\alpha:=\int_{M}{f\,d\pi}$ and $\phi$ as in the statement so that $f-\alpha = \phi - P \phi$. We write 
 
\begin{equation}
\sum_{i=1}^{n}{f(Z_i) - n \alpha} =\sum_{i=1}^{n}{\left[ 
\phi(Z_{i+1}) - P\phi(Z_i)\right]} + \left[ \phi(Z_1)-\phi(Z_{n+1})\right]. \label{decomposition} \end{equation}
On the one hand, the sequence $D_i:=  \phi(Z_{i+1}) - P\phi(Z_i)$ is a martingale difference sequence with respect to the canonical filtration of $(Z_i)_i$. Moreover, $|D_i|\leq 2 \|\phi\|_{\infty}$.  Thus  $M_n:=\sum_{i=1}^n{\left[ \phi(Z_{i+1}) - P\phi(Z_i)\right]}$ is a martingale with bounded differences. Applying Azuma--Hoeffding concentration inequality for martingales with bounded difference (see for instance \cite[Lemma 4.1]{McDiarmid}), we get that for every $t>0$ and $n\in \N$, \begin{equation}\label{azuma}\p\left( M_n \geq nt/2\right)\leq \exp \left(-\frac{nt^2}{32 \|\phi\|^2} \right)\,\,\,\textrm{and}\,\,\,
  \p\left( M_n \leq  -nt/2\right)\leq \exp \left(-\frac{nt^2}{32 \|\phi\|^2} \right).\end{equation}

On the other hand, the following crude upper bound holds for $V_n:=   \phi(Z_1)-\phi(Z_{n+1})$; for every $n\in \N$, we have $|V_n|\leq 2\|\phi\|_{\infty}$. Hence,  $|V_n|\leq nt/2$ for every $n\geq \frac{4\|\phi\|_{\infty}}{t}$. Combining this fact with \eqref{decomposition} and \eqref{azuma}, we get that for every $t>0$ and every $n\geq \frac{4\|\phi\|_{\infty}}{t}$, 

$$\p\left( \sum_{i=1}^n{f(Z_i)} - n\int_{M}{f\,d\pi} \geq nt\right)\leq \exp\left(-\frac{nt^2}{32 \|\phi\|_{\infty}^2}\right) $$ and 
$$
 \p\left( \sum_{i=1}^n{f(Z_i)} - n\int_{M}{f\,d\pi} \leq -nt\right)\leq \exp\left(-\frac{nt^2}{32 \|\phi\|_{\infty}^2}\right).$$
Thus $\p\left( |\sum_{i=1}^n{f(Z_i)} - n\int_{M}{f\,d\pi} | \leq  nt\right)\leq 2\exp \left(-\frac{nt^2}{32 \|\phi\|_{\infty}^2} \right)$. This shows the desired inequality in the case $n \geq \frac{1}{t} 4\|\phi\|_{\infty}$.  Suppose finally that 
$n  \leq \frac{1}{t}  4 \|\phi\|_{\infty}$. In this case, 
 $nt^2 \leq 16\|\phi\|^2 $ and then 
$\exp \left(-\frac{nt^2}{32\|\phi\|_{\infty}^2}\right)
\geq \exp(-1/2)>\frac{1}{2}$. The desired estimate  holds trivially in this case. 
\end{proof}

\section{Applications to random matrix products and random walks on hyperbolic spaces}\label{sec.non.explicit}
 
The goal of this section is to obtain two consequences of Proposition \ref{prop.cocycle} in the settings of random matrix products and random walks on hyperbolic spaces $M$. For the latter, in this section, we will not suppose any properness assumption, and relatedly, we are only able to obtain non-explicit concentration estimates. In \S \ref{sec.explicit}, we will upgrade those to more explicit estimates in the case of proper hyperbolic spaces.

\subsection{Subgaussian concentrations for random matrix products}\label{subsec.matrix}

Let $d \geq 1$ be an integer, we consider $\mathbb{C}^d$ endowed with the canonical Hermitian structure and $M_d(\C)$ with the induced operator norm. For simplicity, we denote by $\|.\|$ both norms on $\C^d$ and $M_d(\C)$. We denote by $X=P(\C^d)$ the projective space of $\C^d$ and we endow it with the standard metric given by
$$\delta([x], [y]):=\frac{\|x\wedge y\|}{\|x\| \|y\|},$$
where the norm $\|\cdot\|$ is the canonical norm on $\bigwedge ^2 \C^d$, $[x]=\C x$ and $[y]=\C y$.

A probability measure $\mu$ on $\GL_d(\C)$ is said to be
(strongly-)irreducible if the support $S$ of $\mu$ does not fix a (finite union of) non-trivial proper subspace(s) of $\C^d$. An irreducible probability measure $\mu$ is said to be proximal if the closure $\overline{\C G_\mu^+}$ in $M_d(\C)$ of the semigroup $G_\mu^+$ generated by the support of $\mu$ contains a rank-one linear transformation.

A probability measure $\nu$ on $X$ is said to be $\mu$-stationary if it is $\mu$-stationary for the Markov operator $P_\mu$ associated to $\mu$. We recall that for a strongly irreducible and proximal probability measure $\mu$ on $\GL_d(\C)$, there exists a unique $\mu$-stationary probability measure $\nu$ on $X$ \cite{Furstenberg.boundary,guivarch-raugi}. We denote by $\mu^*$ the image of $\mu$ under the map $g \mapsto g^*$, where $g^*$ denotes the conjugate-transpose of $\mu$ and by $\nu^*$ the unique-stationary measure of $\mu^*$ (which is also proximal and strongly irreducible).

We denote by $\sigma:\GL_d(\C) \times X \to \R $ the (additive) norm-cocycle given by $\sigma(g,[x])=\ln \frac{\|gx\|}{\|x\|}$.
The solution of the Poisson equation \eqref{cohomological} for the norm cocycle is closely related to regularity properties of the stationary measure $\nu$ on $X$. Indeed 
when $\mu$ has an exponential moment,    \eqref{cohomological}  can be solved using the result of Le Page \cite{page}  establishing a  spectral gap for the Markov operator $P_{\mu}$ acting on  some H\"{o}lder functions of $X$. As proved by Guivarc'h \cite{Guivarch3} this spectral gap property implies the H\"{o}lder regularity of $\nu$. When $\mu$ has a finite second order moment, Benoist--Quint \cite{BQ.CLT.linear} solved the same equation by using and proving the log-regularity of the stationary measure $\nu$. We will rely on their results.


By \cite{BQ.CLT.linear}, the following quantity  
\begin{equation}\label{eq.soln.matrix}
\psi([x]):=-\int{\ln  \frac{\|x\|\,\|y\|}{|\langle x, y \rangle|} \,d\nu^\ast([y])}
\end{equation}
is finite for every $x\in X$ and defines a continuous function $\psi$ on $X$. Moreover, $\psi$ satisfies the cohomological equation 
\begin{equation}\label{cohomological1}\psi - P_{\mu} \psi   =  \phi - \ell(\mu),\end{equation}
where $\phi([v]):=\int {\ln \frac{\|g v\|}{\|v\|}\, d\mu(g)}$, the expected increase at $[v]$. This fact plays the key role in the proof of the following result: 

\begin{proof}[Proof of Proposition \ref{hoeffding-high-dimension}]
We will apply Proposition \ref{prop.cocycle} with $G=\GL_d(\C)$,  $X=P(\C^d)$, and the norm-cocycle $\sigma:G\times X\to \R$. Observe that for every $g\in G$ and $x\in X$,
$$|\sigma(g,x)| \leq \max \left\{\ln \|g\|, \ln \|g^{-1}\| \right\}.$$
Furthermore, the equation \eqref{cohomological1} shows that the hypothesis \eqref{cohomological} of Proposition \ref{prop.cocycle} holds, and consequently, we deduce that for every $t>0$, $n\in \N$, $v\in \C^d\setminus \{0\}$ we have
\begin{equation}
\label{conc-norm-vector}\p\left( \left| \ln \frac{\|L_n v \|}{   \|v\|} - n\ell(\mu)\right|\geq nt \right) \leq 2\exp\left(-\frac{nt^2}{32  (\kappa_S +{\mathfrak{c)}^2}}\right),\end{equation}
where $\mathfrak{c}=\mathfrak{c}(\mu^\ast)=\sup_{[x] \in P(\C^d)}\psi([x])$. This proves the first estimate.  To get the concentration estimates for the matrix norm of $L_n$, consider the canonical basis $e_1, \cdots, e_n$ of $\C^d$. For every $g\in G$, we have 
$$\|g\|\leq \sqrt{d} \max\{\|g e_i\| : i=1, \cdots, d\}.$$

Thus 
\begin{equation}\label{inter}\p \left(  \ln \|L_n\|  - n \ell(\mu) \geq  n t \right) 
\leq d\,  \max_{i=1,\ldots,d}\p \left(  \ln \|L_n e_i\|  - n \ell(\mu) \geq n t - \frac{  \ln d}{2} \right).\end{equation}
Suppose  that $nt \geq   \ln d$. 
Then  $nt - (\ln d)/2 \geq \frac{nt}{2}$ and hence, 
by combining \eqref{conc-norm-vector} and \eqref{inter}, we get that 
 
$$\p\left(\left|  \ln \|L_n  \|    - n\ell(\mu)\right|\geq nt \right) \leq 2 d \exp\left(-\frac{nt^2}{128    (\kappa_S +{\mathfrak{c)}^2}}\right),$$
as claimed. 
\end{proof}

\subsection{Application to random walks on hyperbolic spaces}\label{subsec.hyp} The goal of this part is to deduce concentration estimates for non-elementary random walks on (not necessarily proper) geodesic hyperbolic spaces. 

The main tool is Proposition \ref{prop.cocycle} that we will apply to the horofunction compactification $\overline{M}^h$ and Busemann cocycle $\sigma$ of a separable geodesic hyperbolic metric space.
The key point in this application is to solve the cohomological equation \eqref{cohomological} in this setting. This was previously done by Benoist--Quint \cite{BQ.CLT.hyperbolic} when $M$ is proper; they gave a solution $\psi$ on $\partial_h M$. A partial extension of this solution to $\partial_h M$ was used by Horbez \cite{horbez} in the non-proper setting.
We will observe here that $\psi$ extends further to a solution on the full space $\overline{M}^h$; this will be more convenient for our purpose. 

Let us start by recalling some definitions. Let $(M,d)$ be a separable metric space and denote by $\Lip^1(M)$ the set of real valued Lipschitz functions on $M$ with Lipschitz constant 1, endowed with the  topology of pointwise convergence. Fixing $o \in M$, for $x \in M$, let the function $h_x \in \Lip_{o}^1(M)$, defined by $h_x(m)=d(x,m)-d(x,o)$, where $\Lip_{o}^1(M)$ is the subspace of $\Lip^1(M)$ consisting of functions $f$ satisfying $f(o)=0$. The closure of $\{h_x  : x \in M\}$ is a compact metrizable subset of $\Lip^1_o(M)$, called the \textit{horofunction compactification} of $M$ (see e.g.\ \cite[Proposition 3.1]{maher-tiozzo}). It will be denoted as $\overline{M}^h$. The map $x \mapsto h_x$ is injective on $M$ and we usually identify $M$ with its image in $\overline{M}^h$. The \textit{horofunction boundary} of $M$ is defined as $\partial_h M:=\overline{M}^h\setminus M$. The group of isometries $\Isom(M)$ acts on $\overline{M}^h$ by homeomorphisms given, for $g \in \Isom(M)$,  $h \in \overline{M}^h$ and $m \in M$, by $(g.h)(m)=h(g^{-1}m)-h(g^{-1}o)$. This extends equivariantly the isometric action of $\Isom(M)$ on $M$ and the set $\partial_h M \subset \overline{M}^h$ is invariant under $\Isom(M)$. The \textit{Busemann cocycle} $\sigma: \Isom(M) \times \overline{M}^h \to \mathbb{R}$ is defined by $$\sigma(g,h)=h(g^{-1}o).$$

Now, let $(M,d)$ be, moreover, a $\delta$-hyperbolic space. We recall that this means that for every $x,y,z,o \in M$,
\begin{equation}\label{eq.defining.eq}
(x|y)_o \geq (x|z)_o \wedge (z|y)_o -\delta,    
\end{equation}
where $(.|.)_.$ is the Gromov product given by $(x|y)_o=\frac{1}{2}(d(x,o)+d(y,o)-d(x,y))$. For simplicity, we will often omit the basepoint $o$ from the notation. We refer to \cite{hyperbolic-book} for general properties of these spaces. An element $\gamma \in \Isom(M)$ is said to be loxodromic if for any $x \in M$, the sequence $(\gamma^nx)_{n \in \Z}$ constitutes a quasi-geodesic (see \cite[Ch.\ 3]{hyperbolic-book}). Equivalently, $\gamma$ is loxodromic if and only if it fixes precisely two points $x_\gamma^+,x_\gamma^-$ on the Gromov boundary $\partial M$ of $M$ \cite[Ch.\ 9 \& 10]{hyperbolic-book}. Two loxodromic elements $\gamma_1,\gamma_2$ are said to be independent if the sets of fixed points $\{x^+_{\gamma_i} ,x^-_{\gamma_i}\}$ for $i=1,2$ are disjoint. Finally, a set $S$, or equivalently a probability measure with support $S$, is said to be non-elementary if the semigroup generated by $S$ contains at least two independent loxodromic elements.

For $h_1,h_2 \in \partial_h M$, we set
$$
(h_1 | h_2)_o=-\frac{1}{2} \inf_{m \in M} \left( h_1(m)+h_2(m) \right).
$$
This extends the usual Gromov product on $M$ based at $o \in M$ to $\partial_h M$. We note that $(h_1|h_2)_o=\infty$ if any only if $h_1$ and $h_2$ have the same projection to the Gromov boundary of $M$.


Let now $\mu$ be a non-elementary probability measure on $G$. 
There might exist several $\mu$-stationary probability measures on $\partial_h M$  but they all have the same Busemann cocycle average which is given by the drift of the $\mu$-random walk on $M$ i.e.\ for any $\mu$-stationary $\nu$ on $\partial_h M$, $\iint_{G\times \partial_h M}{\sigma(g,x)\,d\mu(g) d\nu(x)}=\ell(\mu)$ (see \cite[Proposition 3.3]{BQ.CLT.hyperbolic} or \cite[Corollary 2.7]{horbez}). Recall that $\mu$ is said to have a finite first order moment if $\int \kappa(g) d\mu(g)<\infty$ and that the convergence \eqref{eq.kingman} to the drift $\ell(\mu) \in \mathbb{R}$ is ensured under this moment assumption.

When $\mu$ has a finite second order moment (i.e.~ $\int{\kappa(g)^2 d\mu(g)}<+\infty$) and  $M$ is proper, 
Benoist--Quint showed that the function $\psi$ defined on $\partial_h M$ as 
\begin{equation}\label{eq.soln.hyp}
\psi(x):=-2 \int_{\partial_h M} {(x | y)_o \, d\check{\nu}(y)},
\end{equation}
is bounded, measurable, and it satisfies the Poisson equation 
\begin{equation}\label{eq.hyp.poisson}
\psi(x)-P_\mu\psi(x)=\int \sigma(g,x) d\mu(g)-\ell(\mu),
\end{equation}
where $\check{\nu}$ is any stationary probability measure on $\partial_h M$ for $\mu^{-1}$ and $\mu^{-1}$ is the non-elementary probability measure given by the image of $\mu$ by the map $g \mapsto g^{-1}$. 

For our purposes in the sequel, it will be more convenient to consider the action of the Markov operator $P_{\mu}$ on the space of bounded measurable functions defined on the whole compactification $\overline{M}^h$ in the case where $M$ is only a separable and geodesic hyperbolic space. Accordingly, we will verify that the natural extension of the function $\psi$ given by \eqref{eq.soln.hyp} to the space $\overline{M}^h$ yields a solution to the equation \eqref{eq.hyp.poisson}. We summarize these in the next 

\begin{lemma}\label{cohomological.compactification}
Suppose that $\mu$ has finite second order moment. 
The function $\psi: \overline{M}^h \to \mathbb{R}$ defined by 
\begin{equation}\label{eq.soln.hyp1}
\psi(x)=-2 \int_{\overline{M}^h}(x|y)_o d\check{\nu}(y) \end{equation}
is a bounded measurable function that satisfies the equation
\begin{equation}\label{eq.hyp.poisson1}
\psi(x)-P_\mu\psi(x)=\int \sigma(g,x) d\mu(g)-\ell(\mu)
\end{equation}
for every $x \in \overline{M}^h$.
\end{lemma}

The proof requires the following slight extension of \cite[Lemma 2.4]{horbez}:
\begin{lemma}\label{lemma.horbez}
There exists $C>0$ depending only on the hyperbolicity constant $\delta$ of $M$ such that for all $g \in G$ and $x \in \overline{M}^h$, we have
\begin{equation}\label{eq.horbez1}
|(go | gx)_o -\frac{1}{2}(\kappa(g)+\sigma(g,x))| \leq C,
\end{equation}

\begin{equation}\label{eq.horbez2}
 |(go | x)_o -\frac{1}{2}(\kappa(g)-\sigma(g^{-1},x))|\leq C.   
\end{equation}
\end{lemma}

\begin{proof}
When $x \in M$, expanding the definitions, both inequalities are seen to hold with $C=0$. 
To treat the case when $x \in \partial_h M$, we follow  \cite[\S 3.2]{maher-tiozzo} and write the boundary $\partial_h M$ as the union of two $G$-invariant subsets $\partial^\infty_h M$ and $\partial^f_h M $, where $$
\partial_h^\infty M=\{h \in \partial_h M : \inf_{m \in M} h(m)=-\infty\}
$$
and $\partial_h^{f}M$ is defined  similarly by $\inf_{m\in M}{h(m)}>-\infty$. 
In case $x \in \partial_h^\infty M$, the statement is precisely \cite[Lemma 2.4]{horbez}. On the other hand, for $x=h \in \partial_h^f M$, it is not hard to see that for some constant $C>0$ depending only on $\delta$, the horofunction $h$ stays $C$-close to a horofunction $h_y$ for some $y \in M$ chosen in the coarse minimizer of $h$ (see \cite[\S 3.3]{maher-tiozzo}) i.e.\
\begin{equation}\label{eq.describe.finite}
|h(m)-h_y(m)|\leq C    
\end{equation}
for every $m \in M$. Therefore, when $x \in \partial_h^f M$, the inequalities \eqref{eq.horbez1} and \eqref{eq.horbez2} follow from the first case above where $x \in M$.
\end{proof}

\begin{proof}[Proof of Lemma \ref{cohomological.compactification}] 

The proof goes similarly as Benoist--Quint's proof \cite[Propositions 4.2 \& 4.6]{BQ.CLT.hyperbolic}; we indicate only the needed changes. Let $\check{\nu}$ be any $\mu^{-1}$-stationary probability measure on $\overline{M}^h$. 
We first show that $\psi(x)=-2 \int_{\overline{M}^h}(x|y)_o d\check{\nu}(y) $ is a bounded function  on the whole compactification $\overline{M}^h$. Arguing precisely as in the proof of \cite[Proposition 4.2]{BQ.CLT.hyperbolic} (namely, taking $p=2$ in the authors' proof), it suffices to show that there exists $a>0$ such that
\begin{equation}\label{eq.suffices.suffices}
\sum_{n \geq 1}\sup_{x\in \overline{M}^h}\check{\nu}\{y :  (x|y)_o\geq an\}<+\infty.    
\end{equation}
Therefore, we now focus on obtaining \eqref{eq.suffices.suffices}. First, any $\widetilde{\mu}$-stationary probability measure on $\overline{M}^h$ is supported on  $\partial_h^\infty M \subseteq \partial_h M$ (\cite[Proposition 4.4]{maher-tiozzo}) for $\widetilde{\mu} \in \{\mu,\mu^{-1}\}$.
Since by \cite[Corollary 2.7]{horbez}, the Busemann cocycle $\sigma: G\times \overline{M}^h\to \R$ has a unique cocycle average on the boundary $\partial_h M$, it follows that it has a unique cocycle average on all of the compactification $\overline{M}^h$.
Then, using Benoist--Quint's large deviation result for cocycles \cite [Proposition 3.2]{BQ.CLT.linear} applied to the  continuous Busemann cocycle on   the compact metrizable space $\overline{M}^h$, we deduce that for every $t>0$ and $\widetilde{\mu} \in \{\mu,\mu^{-1}\}$, we have
\begin{equation*}
\sum_{n \geq 1}{\sup_{\xi\in \overline{M}^h} {\widetilde{\mu}^{\ast n}\left \{g : |\sigma(g,\xi) - n \ell(\widetilde{\mu})| > n t\right\}}}<+\infty.
\end{equation*}
Now, using Lemma \ref{lemma.horbez} --- by substituting \eqref{eq.horbez1} for \cite[(2.17)]{BQ.CLT.hyperbolic} and \eqref{eq.horbez2} for \cite[(2.16)]{BQ.CLT.hyperbolic} --- and following the same strategy as in the proof of \cite[Lemma 4.5]{BQ.CLT.hyperbolic}, we deduce that there exists a summable sequence $(C_n)$ of constants and a constant $a>0$ such that for every $x,y\in \overline{M}^h$, 
$$
(\mu^{-1})^{\ast n}\{ g : (g y| x)_o \geq a n \}\leq C_n.
$$
By stationary of $\check{\nu}$, we have  $\check{\nu}\{y: (x|y)_o \geq a n\}=\int{ (\mu^{-1})^{\ast n}\{g : (g y|x)_o \geq a n \} d \check{\nu}(y)}$ for every $n \in \N$. This implies \eqref{eq.suffices.suffices} and shows that $\psi$ is bounded.

Finally, we check the Poisson equation \eqref{eq.hyp.poisson1}. We remark that the following key identity
$$\sigma(g, x) = -2(x|g^{-1}y)_o + 2(gx|y)_o+\sigma(g^{-1},y)$$
used by Benoist--Quint holds also true in our setting for every $g\in \Isom(M)$ and  $x,y \in \overline{M}^h$ provided $gx$ and $y$ do not project to the same point of the Gromov boundary $\partial M$. Since the unique $\mu^{-1}$-harmonic measure on $\partial M$ is non-atomic (\cite[Theorem 1.1]{maher-tiozzo}), we deduce   \eqref{eq.hyp.poisson1}  by integrating both sides of the previous identity with respect to $d\check{\nu}(y) d\mu(g)$ and using the fact that $\ell(\mu^{-1})=\ell(\mu)$.
\end{proof}

\begin{remark}
When $\mu$ has a finite exponential moment, one can substitute Maher's result on H\"{o}lder regularity of the harmonic measure \cite[Lemma 2.10]{maher.jlms} (see also \cite[Proposition 2.16]{BMSS}) for large deviation results of Benoist--Quint to prove that the function $\psi$ is bounded.
\end{remark}

\begin{proposition}\label{prop.hyp}
Let $(M, d)$ be a separable, geodesic, $\delta$-hyperbolic space, $o\in M$ and $\mu$ a non-elementary probability measure on the group $G= \Isom(M)$ with countable bounded support $S$. Denoting $\kappa_S=\sup_{g\in S} d(go,o)$ and  $\mathfrak{c}=-\inf_{x \in X}\psi(x)$ (see Lemma \ref{cohomological.compactification}), for every $t>0$ and $n\in \N$, we have
\begin{equation}\label{eq.prop-busemann}
\sup_{\xi\in \overline{M}^h}\p\left(    \left| \sigma(L_n, \xi)    - n\ell(\mu)\right|\geq nt \right) \leq 2  \exp\left(-\frac{nt^2}{32 (\kappa_S +{\mathfrak{c)}^2}}\right).
\end{equation}
In particular,  
\begin{equation}\label{eq.prop-hyp}
 \p\left(\left| \kappa(L_n)    - n\ell(\mu)\right|\geq nt \right) \leq 2  \exp\left(-\frac{nt^2}{32 (\kappa_S +{\mathfrak{c)}^2}}\right).
\end{equation}
\end{proposition}
  
\begin{proof}
In view of the inequality $|\sigma(g,\xi)| \leq \kappa(g)$ true for every $g\in G$ and $\xi\in \overline{M}^h$, the estimate \eqref{eq.prop-busemann} follows directly from  Proposition \ref{prop.cocycle} applied to $G=\langle \supp(\mu) \rangle$ the group generated by the support of $\mu$ endowed with the discrete topology, $X=\overline{M}^h$ and the Busemann cocycle $\sigma: G \times X \to \mathbb{R}$. Finally, \eqref{eq.prop-hyp} follows directly by specializing to $\xi=o$.  
\end{proof}

\begin{remark}[Proper case]\label{rk.proper.case}
The countability assumption in Proposition \ref{prop.hyp} is not needed when $M$ is proper. Indeed, it is known that in that case the full isometry group $G=\Isom(M)$ is Polish and hence we may apply Proposition \ref{prop.cocycle} for any (non-elementary) Borel probability measure $\mu$ on $G$ with bounded support.
\end{remark}

\section{Explicit estimates for random walks on proper hyperbolic spaces}\label{sec.explicit}

In \S \ref{subsec.main}, we prove our main result on concentration inequalities around the drift for random walks on proper hyperbolic spaces $M$. Exploiting the locally compact structure of $\Isom(M)$, the proof makes crucial use of the harmonic analytic and geometric approach and results of Benoist--Quint \cite[\S 5]{BQ.CLT.hyperbolic}. Respectively in \S \ref{subsec.log.track} and \S \ref{subsec.continuity}, we discuss the Frostman property of the harmonic measure and the continuity properties of the drift.

\subsection{Main result on concentrations}\label{subsec.main}

Let $(M,d)$ be a proper metric space, we denote by $G$ its group of isometries. It is a locally compact group \cite[Theorem 6]{Gao.Kechris} and we denote by $\mu_G$ a Haar measure on $G$. For a probability measure $\mu$ on $G$, $S$ denotes the support of $\mu$ which is the smallest closed subset whose $\mu$-mass equals one. We recall that for every $r \in [0,1)$, we denote by $\mu_{r,\lazy}=r \delta_{\id} + (1-r)\mu$. Having fixed a basepoint $o \in M$,  for an element $g \in \Isom(M)$, we write $\kappa(g)=d(go,o)$ and for a bounded set $S$, we set $\kappa_S:=\sup\{\kappa(g) : g\in S\}$.

The main result of this section is the following result, which immediately implies Theorem \ref{thm.main.with.lambda} by specializing to $\xi=o$. 
\begin{theorem}\label{thm.main.text}
Let $(M,d)$ be a proper geodesic hyperbolic space and $o\in M$. Assume that the group $\Isom(M)$ acts cocompactly on $M$. Then, there exists an explicit positive function $D(.,.)$ with $D(.,\lambda)<\infty$ for every $\lambda \in (0,1)$ such that for every  non-elementary probability measure $\mu$ on $G$ with bounded support $S$, for every $\xi \in \overline{M}^h$, $t>0$ and $n \in \mathbb{N}$, we have  
$$\p\left( |\sigma(L_n,\xi) - n \ell(\mu)|\geq nt \right) \leq 2
\exp\left(\frac{-nt^2}{ \kappa_S^2 D(\kappa_S,\|\lambda_G(\mu_{r,\lazy})\|_2)}\right)$$
for every $r \in [0,1)$.
\end{theorem}

With Proposition \ref{prop.hyp} at hand, the main ingredient for the proof of the Theorem \ref{thm.main.text} is the following

\begin{proposition}\label{quantitative-log-regularity}
Let $(M,d)$ be a proper geodesic $\delta$-hyperbolic space such that the group $G$ of isometries of $M$ acts cocompactly on $M$. Then, there exists an explicit positive function $C(.,.)$ with $C(.,\lambda)<\infty$ for every $\lambda \in (0,1)$ such that for every boundedly supported non-elementary probability measure $\mu$ on $G$, we have
\begin{equation}\label{eq.prop.regularity}
\sup_{y\in \overline{M}^h } \int_{\partial_h M} {(x\,| y)_o\,d\nu(x)}\leq C(\kappa_{S},\|\lambda_G(\mu_{r,\lazy})\|_2).
\end{equation}
for every $r \in [0,1)$.
\end{proposition}

\begin{remark}\label{rk.explicit.C}
1. The function $C(\kappa,\lambda)$ is a function of $\kappa>0$ and $\lambda \in (0,1]$ and, for $\lambda<1$, it can be given by 
$$
\kappa \inf_{1<c<\lambda^{-1/2}} \left[4\left(\frac{\ln \kappa}{\ln c} \vee \frac{1}{(\ln c)^2}\right)+  \frac{A_0}{1-c^2 \lambda}\right],
$$
where $A_0=\left(\frac{\mu_G(B_{2R(\delta)+2D_0})}{\mu_G(B_{R(\delta) +D_0})}\right)^{1/2}$ with $R(\delta)=14\delta+4$, for $r \geq 0$, $B_r:=\{g \in G : d(go,o) \leq r\}$ and $D_0$ is the diameter $2\sup_{x,y\in M}\inf_{g\in G}{d(gx,y)}$, which is finite by the cocompactness assumption. We set $C(\kappa,1)=\infty$. Finally, for concreteness, for $\lambda<1$, specializing to $c=(1+\lambda^{-1/2})/2$, one can get 
$$ C(\kappa,\lambda) \leq \kappa \left( 8\ln^+ (\kappa)+4A_0/3+16 \right) \frac{1}{(1-\sqrt{\lambda})^2}.$$
2. One can obtain a version of \eqref{eq.prop.regularity} with finite first order moment assumption replacing $\kappa_S$ with the first order moment. However, we will not need this more general version.
\end{remark}

\begin{remark}[Why we also consider lazy random walks]\label{rk.lazy.discussion}
Denoting by $\rho(\mu)$ the spectral radius of $\lambda_G(\mu)$, we have $\|\lambda_G(\mu)\|_2=\sqrt{\rho(\mu \ast \mu^{-1})}$, where $\mu^{-1}$ the image of $\mu$ by the map $g \mapsto g^{-1}$. When $\mu$ is symmetric, $\|\lambda_G(\mu)\|_2=\rho(\mu)$. Moreover, thanks to \cite[Theorem 4]{berg-christensen}, $\|\lambda_G(\mu)\|_2<1$ as soon as $\mu \ast \mu^{-1}$ is non-elementary and hence in this case, we may take $r=0$ on the right-hand-side of the the inequality given by the previous result. Finally, for every non-elementary probability measure $\mu$, for every $r>0$, $\mu_{r,\lazy}$ satisfies the previous property and hence for $r \in (0,1)$, $\|\lambda_G(\mu_{r,\lazy})\|_2<1$. We refer to \cite[p177]{berg-christensen} for an example of a non-elementary probability measure $\mu$ with $\|\lambda_G(\mu)\|_2=1$.
\end{remark}


For the proof, we require the following version of \cite[Lemma 5.2]{BQ.CLT.hyperbolic} where we highlight  the constants that appear in the aforementioned lemma for our purposes. This is the crucial harmonic analytic ingredient of the proof where the additional hypothesis (compared to Proposition \ref{prop.hyp}) on the cocompactness action of $G$ on the space $M$ is used.

\begin{lemma}\label{lemma.bq.5.2}
Let $(M,d)$ be a proper metric space. Suppose that the group $G$ of isometries of $(M,d)$ acts cocompactly on $M$. Then, there exists a constant $D_0>0$ depending only on $M$ such that for every probability measure $\mu$, for every $R>0$, $n \geq 1$ and $m,m' \in M$, we have
\begin{equation}\label{estime.1}
\mathbb{P}(d(R_n m,m') \leq R-D_0) \leq A_0(R) \|\lambda_G(\mu)\|_2^n,
\end{equation}
where $A_0(R)=\left(\frac{\mu_G(B_{2R})}{\mu_G(B_{R})}\right)^{1/2}$.
\end{lemma}
\begin{proof}
The proof follows similarly as \cite[Lemma 5.2]{BQ.CLT.hyperbolic}, we only indicate the necessary modifications.\\
\textbullet ${}$ We replace the estimate $\|\lambda_G(\mu)^n\| \leq C_0a_0^n$ used in the proof of \cite[Lemma 5.2]{BQ.CLT.hyperbolic}, by $\|\lambda_G(\mu)^n\| \leq \|\lambda_G(\mu)\|^n$, and consequently, the constants $C_0$ and $a_0$ in \cite[Lemma 5.2]{BQ.CLT.hyperbolic} can, respectively, be taken to be $1$ and $\|\lambda_G(\mu)\|$.\\
\textbullet ${}$ Regarding the constant $A_0$ in \cite[Lemma 5.2]{BQ.CLT.hyperbolic}, in their proof, Benoist--Quint assume that $m$ and $m'$ belongs to the same $G$-orbit. However, since the action of $G$ on $M$ is cocompact, there exists $D_0$ such that for every $x,y$, there exists $g \in G$ with $d(gx,y) \leq D_0/2$ and this allows us to take $A_0=\left(\frac{\mu_G(B_{2R})}{\mu_G(B_{R})}\right)^{1/2}$.
\end{proof}

We will need the following geometric lemma, which is an adaption of Inclusion (5.5) in the proof of  \cite[Lemma 5.3]{BQ.CLT.hyperbolic} to the horofunction compactification. We point out that this is the key point where the geometric assumption of hyperbolicity is used. 

\begin{lemma}\label{lemma.BQ.key}
Let $(M,d)$ be a proper geodesic hyperbolic metric space and $o \in M$. Then, there exists $R(\delta)>0$ such that for any $\xi \in \partial_h M$, $y\in \overline{M}^h$, and constant $D > 0$, there exists a finite subset $C\subset M\times M$ with at most $D^2$ elements such that for any $g\in \Isom(M)$, there exists $(x',y')\in C$ such that
$$(g \xi| y)_o\geq \kappa(g) \Longrightarrow [\kappa(g)\geq D]
\vee [d(gx', y')\leq R(\delta)].$$
Moreover, all elements constituting the tuples in $C$ are contained in a ball of radius $D$ around $o$.
\end{lemma}

As it is will be shown, one can take the constant $R(\delta)=14\delta +4$.

The proof will require some juggling between  horofunction boundary $\partial_h M$ and Gromov boundary $\partial M$ to construct the set $C$. We therefore start by recalling some standard facts on the relation between $\partial_h M$ and $\partial M$.

First, there exists a natural $G$-equivariant surjective map from $\overline{M}^h$ to $M \cup \partial M$. Namely, given $h \in \partial_h M$, for any sequence $x_n \in M$ such that $h_{x_n} \to h$, the sequence $x_n$ Gromov converges to infinity in the sense that $\inf_{m,n \geq k} (x_n,x_m) \to \infty$ as $k \to \infty$. When we endow $M \cup \partial M$ with the usual topology \cite[Ch.\ 2]{hyperbolic-book}, this projection is the unique map that continuously extends the identity map $m \mapsto m$ on $M$. For $y \in \overline{M}^h$, we denote by $\pi_y$ its image in $M \cup \partial M$. 

Given two points $\pi_x \neq \pi_y \in \partial M$, by using the defining inequality \eqref{eq.defining.eq} of a $\delta$-hyperbolic space, one checks that for any pair of pair of sequences $x_n,x'_n$ and $y_n,y_n'$ that Gromov converge to infinity and that are, respectively, in the equivalence class of $\pi_x$ and $\pi_y$, we have 
\begin{equation}\label{eq.upto.2delta}
\limsup_{n \to \infty} (x'_n|y'_n)_o -\liminf_{n \to \infty}(x_n|y_n)_o\leq 2\delta.
\end{equation}

For reader's convenience, we single out two basic geometric properties that are used in the proof of the previous lemma. 

\begin{lemma}[Thin triangles]\label{lemma.tt}
Given a triple $(x,y,z)$ with $x,y \in M$ and $z \in \partial M$, fix geodesic rays $[x,y], [x,z]$ and $[y,z]$. Then, there exist $a \in [x,y]$, $b \in [x,z]$ and $c \in [y,z]$ 
with the following property: denoting the corresponding geodesic rays oriented from $i$ to $j$ by $\gamma_{ij}$, where $i,j\in \{x,y,z,a,b,c\}$, we have for every $t \geq 0$ in the respective interval of definition, all distances $d(\gamma_{xb}(t),\gamma_{xa}(t)),d(\gamma_{ya}(t),\gamma_{yc}(t))$ and $d(\gamma_{bz}(t),\gamma_{cz}(t))$ are $\leq 6\delta$. 
\end{lemma}

This lemma can be deduced from standard facts in hyperbolic geometry. We include a brief proof for reader's convenience.

\begin{proof}
Consider the triangle $(x,y,z)$ whose edges are as given in the statement. Let $z_n$ be a sequence of points on the edge $(y,z)$ that converge to $z$. Consider the segments $\zeta_n$ from $x$ to $z_n$. Since $M$ is proper, by Arzel\`{a}--Ascoli Theorem, up to subsequence, they converge to a ray $\zeta$ between $x$ and $z$.

For each triangle $(x,y,z_n)$, fix points $a_n,b_n,c_n$ respectively on the edges $[x,y]$, $[x,z_n]$ and $[y,z_n]$ that map to the junction point of the associated tripod \cite[Ch.\ 1]{hyperbolic-book}. Using the fact that $M$ is proper and passing to a further subsequence of $z_n$, we may suppose that the sequences $a_n,b_n,c_n$ converge, respectively, to the points, $a\in [x,y], b' \in \zeta$ and $c \in [y,z]$. Let $b$ be the point on $[x,z]$ at distance $d(x,a)$ from the $x$. 

Now note that by the tripod lemma \cite[Proposition 3.1]{hyperbolic-book}, we have the required property within each triangle $(x,y,z_n)$ with $4\delta$. Since all points $a_n,b_n,c_n$ converge to respectively $a,b',c$, the same property is true at the limit triangle with $[x,z]$ replaced by $\zeta$. Now since $[x,z]$ and $\zeta$ are at parametrized-distance $2\delta$-apart, we get the required property with $6\delta$.
\end{proof}


\begin{lemma}[Fellow travellers]\label{lemma.fellow.travel}
For every $\xi, \eta \in \overline{M}^h$, let $\gamma_\xi$ and $\gamma_\eta$ be 
geodesic rays such that $\gamma_\xi(0)=\gamma_\eta(0)=o$ and $\gamma_\zeta(t) \to \pi_\zeta$ as $t \to \infty$ for $\zeta \in \{\xi,\eta\}$. Then for any $r \geq 0$ such that $(\xi|\eta)_o \geq r$, we have $d(\gamma_\xi(t), \gamma_\eta(t)) \leq 8\delta$ for every $t \in [0,r]$.
\end{lemma}
\begin{proof}
We can suppose that $\pi_\xi \neq \pi_\eta$ and $r \geq 2\delta$. It follows by \eqref{eq.upto.2delta} that for every $\epsilon>0$ for every $s>0$ large enough, we have $(\gamma_\xi(s) | \gamma_\eta(s))_o \geq r-2\delta-\epsilon$. The statement now follows by expanding the inequality
$(\gamma_\xi(t)| \gamma_\eta(t))_o \geq (\gamma_\xi(t), \gamma_\xi(s)) \wedge (\gamma_\xi(s),\gamma_\eta(s))\wedge (\gamma_\eta(s),\gamma_\eta(t))-2\delta$
for $t \leq r$ and $s \to \infty$.
\end{proof}

We now give the
\begin{proof}[Proof of Lemma \ref{lemma.BQ.key}]
We will prove the claim with $R=14\delta +4$. Let such $\xi \in \partial_h M$, $y \in \overline{M}^h$ and $D>0$ be given. Let $g\in \Isom(M)$. To construct the set $C$, fix three rays $[o,\pi_\xi]$, $[o,\pi_y]$ and $[o,g\pi_\xi]$ and for $i=1,\ldots,\lfloor D \rfloor-1$, let $m_i$, $m_i'$ and $m_i''$ be points  on the respective rays satisfying $d(z_i,o)=i \wedge d(o,\zeta)$ for every couple  $(z_i,\zeta) \in \{(m_i,\pi_\xi), (m_i',\pi_y), (m_i'',g\pi_\xi)\}$. We denote $m_0=m'_0=m_0''=o$. 
Suppose now that $(g\xi|y)_o \geq \kappa(g)$ and $\kappa(g)<D$. 
Using the thin triangles Lemma \ref{lemma.tt} for $(o,go,g\xi)$, since $D >\kappa(g)$, we deduce that there exists $i_0, j_0 \leq \kappa(g)$ such that $d(gm_{i_0},m_{j_0}'') \leq 6 \delta +4$.

We now use the fellow-travellers Lemma \ref{lemma.fellow.travel} for $g\xi,y \in \overline{M}^h$ with the geodesic rays $[o,g\pi_\xi]$ and $[o,\pi_y]$. Since $(g\xi|y)_o \geq r:=\kappa(g)$ and $j_0 \leq \kappa(g)$, we find that there exists $i_1 \leq \kappa(g)$ such that $d(m''_{j_0}, m'_{i_0}) \leq 8\delta$.

We deduce that $d(gm_{i_0},m_{i_1}') \leq 14\delta +4$. Therefore, desired result holds with $C:=\{(m_i, m'_j)| 0\leq i,j\leq \lfloor D \rfloor -1\}$.
\end{proof}

\begin{proof}[Proof of Proposition \ref{quantitative-log-regularity}]
We fix $r \in [0,1)$ such that $\|\lambda_G(\mu_{r,\lazy})\|<1$. The latter may be equal to 1 only for $r=0$ (see Remark \ref{rk.lazy.discussion}) in which case the inequality holds trivially by setting $C(.,1)\equiv \infty$. 
Note that the measure $\nu$ is $\mu_{r,\lazy}$-stationary and denoting by $S_r$ the support of $\mu_{r,\lazy}$, we have $\kappa_{S_r}=\kappa_S$. To ease the notation, in the proof, we write $\mu$ for $\mu_{r,\lazy}$.

Let $\nu$ be a $\mu$-stationary probability measure on $\partial_h M$. Let $(B=G^{\N}, \beta=\mu^{\otimes \N})$ be the Bernoulli space and $T:B \to B$ the shift map.  Since $\partial_h M$ is compact, metrizable (see e.g.\ \cite[Proposition 3.1]{maher-tiozzo}) and $\Isom(M)$ acts continuously on $\partial_h M$, by a result of Furstenberg \cite{Furstenberg.boundary}, it follows that for $\beta$-almost every $b\in B$, there exists a probability measure $\nu_b$ on  $\partial_h M$ such that the following weak convergence holds  
\begin{equation}(b_1\cdots b_n)\ast \nu \overset{\textrm{weakly}}{\underset{n\to +\infty}{\longrightarrow}} \nu_b.\label{ee1}\end{equation} Moreover, for every $n \in \N$, we have 
\begin{equation} (b_1\cdots b_n)\ast \nu_{T^n b} =\nu_b \qquad \text{and} \qquad \nu=\int_{B} \nu_b d\beta(b).\label{ee2}\end{equation}
For every $b\in B$ and $n\in \N$, denote for simplicity $R_n(b):=b_1 \cdots b_n$ and  $k_n(b):=\kappa(R_n(b))$. Let now $\eta\in \overline{M}^h$.
Using \eqref{ee2} and Fubini--Tonelli, we have 
\begin{align*}
& \int_{\partial_h M} (\eta|\xi)_o d\nu(\xi)= \int_B \int_{\partial_h M} (\eta|y)_o d\nu_b(y) d\beta(b) \\ 
&= \int_B \int_0^\infty \nu_b((\eta|y)_o \geq t) dt d\beta(b) 
\\ &
\leq
 \int_B \sum_{n=0}^\infty \int_0^\infty 1_{\kappa_n(b)\leq t <\kappa_{n+1}(b)} \nu_b((\eta|y)_o\geq t) dt  d\beta(b) ,
\end{align*}
where we used the fact that $\kappa_n(b)\to +\infty$ almost surely and where, for every $t>0$, we set $1_{\kappa_n(b)\leq t <\kappa_{n+1}(b)}=0$ whenever $\kappa_n(b)\geq \kappa_{n+1}(b)$. 
Using  Fubini--Tonelli's theorem, we deduce that 
\begin{align*}
\int_{\partial_h M} (\eta|\xi)_o d\nu(\xi)&=
\sum_{n=0}^\infty \int_B \int_{0}^{\infty} \nu_b((\eta|y)_o \geq t)\mathbf{1}_{\kappa_n(b)\leq t < \kappa_{n+1}(b)} dt d\beta(b) \\ & \leq
\sum_{n=0}^\infty \kappa_S \int_B \nu_b((\eta|y)_o \geq \kappa_n(b)) d\beta(b) 
\end{align*}
Now using the first equality of (\ref{ee2}), we have  
\begin{equation*}\begin{aligned}
& \int_{\partial_h M} (\eta|\xi)_o d\nu(\xi) \leq 
\sum_{n=0}^\infty \kappa_S \int_B \nu_{T^n b}\left((\eta| b_1 \cdots b_n y)_o \geq \kappa_n(b)\right) d\beta(b) \\ &=
\sum_{n=0}^\infty \kappa_S \int_B \int_G \nu_{b}((\eta|gy)_o \geq \kappa(g)) d\mu^{\ast n}(g) d\beta(b)\\
&=
\sum_{n=0}^\infty \kappa_S \int_B \int_{\partial_h M} \mu^{\ast n}((\eta|gy)_o \geq \kappa(g)) d\nu_{b}(y) d\beta(b),
\end{aligned}
\end{equation*}
where in the second line we used the fact that $\beta=\mu^{\otimes \N}$ is a product measure and in the last line we used Fubini--Tonelli's theorem. 
We conclude that 
$$
\int_{X}{(\eta|x)_o\,d\nu(x)} \leq  \sum_{n=0}^\infty \kappa_S \sup_{y\in \partial_h M}\p\left((\eta| R_n y)_o \geq \kappa(R_n)\right).
$$
Using now Lemma \ref{lemma.BQ.key}, we get that  for every $c>1$, $n \in \mathbb{N}$, $\xi \in \partial_h M$, $y\in \overline{M}^h$, 
$$\p \left((R_n \xi | y)_o\geq \kappa(R_n)\right)\leq 
\underset{a_n}{\underbrace{\p(\kappa(R_n)\geq c^n)}} + \underset{b_n}{\underbrace{c^{2n} \sup_{x',y'\in  M}\p( d(R_n x', y')\leq R(\delta))}}.$$

Thus
\begin{equation}\label{estime.integral}\int_{X}{(\eta|x)_o\,d\nu(x)}\leq \kappa_S \left( \sum_{n=0}^{+\infty}{a_n} + \sum_{n=0}^{+\infty}{b_n}\right).\end{equation}
 
On the one hand, since $\kappa(R_n)\leq n \kappa_S$, 
we have 
$\sum_{n=0}^{+\infty}{a_n}
\leq \sum_{n=0}^{+\infty}{\mathbf{1}_{c^n \leq n \kappa_S}}$.
Using this, it is not hard to deduce that
\begin{equation}\label{eq.born.an}
\sum_{n=0}^{+\infty}{a_n}
\leq \max \left\{\frac{2\ln \kappa_S}{\ln c} , \frac{4}{(\ln c)^2} \right\}.
\end{equation}

On the other hand, by Lemma \ref{lemma.bq.5.2}, for every $n\in \N$, we have for every $x',y'\in M$, 
$$ \p( d(R_n x', y')\leq R(\delta)) \leq A_0(R(\delta) +D_0) \|\lambda_G(\mu)\|_2^n,$$
where $A_0(.)$ is the function defined in that lemma.

We deduce that for every $1<c<\|\lambda_G(\mu)\|_2^{-1/2}$, 
\begin{equation}\label{eq.born.bn}
\sum_{n=0}^{+\infty}{b_n} \leq A_0(R(\delta) +D_0) \sum_{n=0}^{+\infty}{ (c^2 \|\lambda_G(\mu)\|_2)^n} \leq \frac{A_0(R(\delta) +D_0)}{1-c^2 \|\lambda_G(\mu)\|_2}.
\end{equation}
The proof follows by combining \eqref{estime.integral}, \eqref{eq.born.an} and \eqref{eq.born.bn}. 
\end{proof}
 

 

Theorem \ref{thm.main.text} now directly follows by putting together Propositions \ref{prop.hyp} and \ref{quantitative-log-regularity}.

\begin{proof}[Proof of Theorem \ref{thm.main.text}]
Using the estimate \eqref{eq.prop.regularity} in combination with Lemma \ref{cohomological.compactification}, one gets that in Proposition \ref{prop.hyp} (see also Remark \ref{rk.proper.case}), the constant $\mathfrak{c}$ is bounded above by $2C(\kappa_S,\|\lambda_G(\mu_{r,\lazy})\|_2)$. Since the right-hand-side of the inequality \eqref{eq.prop-hyp} is increasing in $\mathfrak{c}$, we can substitute $2C(\kappa_S,\|\lambda_G(\mu_{r,\lazy})\|_2)$ for $\mathfrak{c}$, one gets that for every $\xi \in \overline{M}^h$, $n \in \N$ and $t>0$, 
$$
\mathbb{P}(|\sigma(L_n,\xi) - n \ell(\mu)|\geq nt) \leq 2   \exp \left(-\frac{n t^2}{32(\kappa_S + 2C\left(\kappa_S,\|\lambda_G(\mu_{r,\lazy})\|_2)\right)^2}\right)
$$
for every $r \in [0,1)$. This yields the desired estimate.
\end{proof}

\begin{remark}[\textit{Random walks with unbounded support}]\label{rk.unbounded.support}
As mentioned in the introduction (Remark  \ref{rk.as.mentioned}), one can have a version of Theorem \ref{thm.main.text} under the assumption that the probability measure $\mu$ driving the random walk has a finite exponential moment, i.e.~ there exists $\alpha_0>0$ such that $\int e^{\alpha_0 \kappa(g)} d\mu(g) < \infty$. Indeed, the use of Azuma--Hoeffding concentration inequality in \S \ref{sec.cocycle} can be replaced for example by the result of Liu--Watbled \cite[Theorem 1.1]{liu-watbled} adapted for martingale differences with conditionally bounded exponential moment. 
Using the latter, one obtains the following version of Theorem \ref{thm.main.text}. Keep the assumptions of Theorem \ref{thm.main.text} and (instead of the bounded support assumption) suppose that there exists $\alpha_0>0$ such that $\int e^{\alpha_0 \kappa(g)} d\mu(g):=K<\infty$. Then, there exists a positive real $c$ such that for every $\xi \in \overline{M}^h$, $t>0$ and $n \in \mathbb{N}$, we have  
$$\p\left( |\sigma(L_n,\xi) - n \ell(\mu)|\geq nt \right) \leq \begin{cases}
      2e^{-nct^2}, & \text{if}\ t \in [0,1] \\
      2e^{-nct}, & \text{otherwise}
    \end{cases}
$$
The constant $c$ depends only on $\alpha_0$, $K$ and  on the constants $A_0$, $R_\delta$, $D_0$ and $\|\lambda_G(\mu_{1/2,\lazy})\|_2$ as in Theorem \ref{thm.main.text}. 
This statement has the obvious advantage of applying to random walks with unbounded support (with finite exponential moment) but it has the disadvantage that the dependence of the appearing constant $c$ to aforementioned parameters of $\mu$ is considerably more complicated. In line with our goals in this article, we have chosen not to give more details on this version of Theorem \ref{thm.main.text} for finite exponential moment random walks. 


\end{remark}

In the remainder of this section, we will single out two applications of the methods we used in this part. Namely, in \S \ref{subsec.log.track}, we will give an explicit bound for the bottom of the support of Hausdorff spectrum of the harmonic measure, and in \S \ref{subsec.continuity}, we discuss an application to continuity of the drift.

\subsection{Application to the Frostman property of the harmonic measure}\label{subsec.log.track}

In this part, we keep the assumptions of Theorem \ref{thm.main.text}. In particular, $\mu$ is a non-elementary probability measure with finite support on $\Isom(M)$ where $M$ is a proper geodesic hyperbolic metric space.

Let $B(x,r)$ denote the ball of radius $r$ around $x$ for a natural metric coming from the Gromov product (we do not go into the details here since this metric will not be used, see \cite[\S 8]{GdH}, \cite[Proposition 5.16]{vaisala}). The following result provides an explicit constant $s_0>0$ for which the Frostman type property
$$\nu(B(x,r)) \leq Cr^{s_0}$$
holds for some constant $C>0$ and every $x \in \partial M$ and $r \geq 0$. 

Such a constant gives a lower bound for the bottom of the support of Hausdorff spectrum of $\nu$; see the work of Tanaka \cite{tanaka.spectrum} who gives a thorough multifractal analysis of the harmonic measure in the special setting of hyperbolic groups. Finally, we mention that the existence of such a (non-explicit) constant for the harmonic measure is known in a more general setting of (not necessarily proper) hyperbolic spaces (see \cite[Lemma 2.10]{maher.jlms} and \cite[Corollary 2.17]{BMSS}).

\begin{proposition}[Frostman property]\label{prop.frostman}
Under the hypotheses of Theorem \ref{thm.main.text}, for 
every $$s< \sup_{r\in [0,1)} \frac{1}{\kappa_S}\ln \left( \frac{1}{\|\lambda_G(\mu_{r,\lazy}))\|_2}\right),$$
we have
\begin{equation}\label{reg}\sup_{x\in M \cup \partial M}
\int_{\partial M} {e^{s\, (x|y)_o}\, d\nu(y)}=:K(\mu,S)<+\infty.\end{equation}
In particular, for each such $s>0$, for every $x\in M$, 
\begin{equation}\label{frostman}\nu \left\{y\in \partial_h M : (x|y)_o\geq t \right\} \leq K(\mu,s) e^{-st}.\end{equation}
\end{proposition}

In the above statement and hereafter, for $x,y \in M \cup \partial M$, the Gromov product $(x|y)_o$ is defined as $\inf \liminf_{n \to \infty}(x_n|y_n)_o$, where the infimum is taken over all sequences $x_n$ and $y_n$ that converge, respectively, to $x$ and $y$.

Since the proof of the above result follows similar lines as the proof of Proposition \ref{quantitative-log-regularity}, we will keep its notation and content with indicating the main lines and changes. We write $B=G^\N$, $\beta=\mu^{\otimes \N}$. By $T$ we denote the shift map on $B$ and for $b \in B$ and $n \in \N$, we write $R_n(b)=b_1\ldots b_n$. It is well-known that for $\beta$-a.s.\ $b \in B$, the sequence $R_n(b)o$ converges to an element of the Gromov boundary $\partial M$ (see \cite[Proposition 3.1.b]{BQ.CLT.hyperbolic} or \cite[Theorem 1.1]{maher-tiozzo}). This defines a boundary map 
$$\xi:B \to \partial M, \quad \text{given by} \quad \xi_b=\lim_{n \to \infty}R_n(b)o$$ on a $T$-invariant $\beta$-full measure subset $B'$ of $B$ which does not depend on $o \in M$. In what follows, we will ignore the distinction between $B'$ and $B$, this should not cause confusion. The fact that in the following proof (as opposed to Proposition \ref{quantitative-log-regularity}) we are working over the Gromov boundary $\partial M$  brings additional simplifications (cf.~ \cite[Prop 5.1]{BQ.CLT.hyperbolic}); recall that for the $\mu$-harmonic measure $\nu$ on $\partial M$, for $\beta$-a.s.\ the limit measure $\nu_b=b_1\ldots b_n \nu$ is a Dirac mass and we have $\nu_b=\delta_{\xi_b}$.

\begin{proof}
Let $s>0$. The beginning of the proof of Proposition \ref{quantitative-log-regularity} is replaced by the following series of identities obtained by multiple applications of Fubini--Tonelli's theorem: We have 
\begin{align*}
& \int_{\partial M} \exp(s\,  (\eta|\xi)_o) d\nu(\xi)= \int_B \int_{\partial M} \exp(s\,(\eta|\xi_b)_o ) d\beta(b) \\ &= 
\int_B \int_0^\infty 1_{(\eta|\xi_b)_o \geq \frac{\ln t}{s}} dt d\beta(b)\\
& =s  \int_B \int_{-\infty}^\infty \exp(s t) 1_{(\eta|\xi_b)_o \geq t} dt d\beta(b)\\
& 
\leq 
1+s \sum_{n=0}^\infty  \int_B \int_{0}^{\infty}   \exp(st) 1_{(\eta|\xi_b)_o \geq t} \mathbf{1}_{\kappa_n(b)\leq t < \kappa_{n+1}(b)} dt d\beta(b) \\ &   \leq 1+s
\sum_{n=0}^\infty \kappa_S \exp(s(n+1)\kappa_S) \int_B 1_{(\eta|\xi_b)_o \geq \kappa_n(b)} d\beta(b).
\end{align*}
Now, let $c>1$ and denote by $a_n$ and $b_n$ the sequences (that depend on $c$) as in the proof of Proposition \ref{quantitative-log-regularity}.
Then 
$$\int_{\partial M} \exp(s\,  (\eta|\xi)_o) d\nu(\xi)\leq 1+ s \kappa_S\sum_{n=0}^{\infty}{ a_n \exp(s(n+1)\kappa_S) } + s \kappa_S \sum_{n=0}^{\infty}{b_n \exp(s(n+1) \kappa_S)}.$$
The sum in the middle of the right hand side is a finite sum. The last series is finite 
as soon as 
\begin{equation}\label{eq.choose.c}
\sum_{n=0}^{+\infty}{ \left(c^2 \exp(s\kappa_S) \|\lambda_G(\mu)\|_2 \right)^n}
\end{equation}
is finite. Now, for any $s<\frac{1}{\kappa_S}\ln \frac{1}{\|\lambda_G(\mu)\|_2}$, one can choose $c>1$ so that \eqref{eq.choose.c} is finite and this finishes the proof of \eqref{reg}. Finally \eqref{frostman} follows from \eqref{reg} by Markov inequality. 
\end{proof}

\subsection{Applications to continuity of the drift}\label{subsec.continuity}
A consequence of Theorem \ref{thm.main.with.lambda} is a uniform control, over different driving probability measures with controlled parameters $\kappa_S$ and $\|\lambda_G(\mu)\|_2$, of large deviations of the displacement around the drift. In turn, this allows one to deduce that the drift varies continuously when one perturbs $\mu$ in such a way that that $\kappa_S$ remains bounded and $\|\lambda_G(\mu)\|_2$ remains away from $1$, as we show in Corollary \ref{corol.continuity} below. The idea of such a deduction, of continuity from uniform large deviations, already appears in the literature, see e.g.\ Duarte--Klein \cite[Ch.\ 3]{duarte-klein.book}. However, with our method, this way of deducing the continuity is not optimal as  one can deduce a continuity result directly from unique cocycle-average property for the Busemann cocycle (which was a key point in obtaining our concentration result). We refer to  Proposition \ref{prop.continuity} for a general continuity statement.

\begin{corollary}\label{corol.continuity}
Let $(M,d)$ be a proper geodesic hyperbolic metric space such that $\Isom(M)$ acts cocompactly on $M$. Consider a sequence of non-elementary probability measures $(\mu_m)_{m\in \N}$ with bounded support $S_m$ in the group $\Isom(M)$ such that $$\limsup_{m \to \infty} \inf_{r \in [0,1)} \|\lambda_G((\mu_m)_{r,\lazy})\|_2<1 \qquad \text{and} \qquad \sup_{m \in \mathbb{N}}\kappa_{S_m}<\infty.$$ Suppose that $\mu_m$ converges weakly to some probability measure $\mu_{\infty}$. Then, as $m \to \infty$ $$\ell(\mu_m) \to \ell(\mu_\infty).$$
\end{corollary}

\begin{proof}
Fix $t_0>0$ and let $\lambda<1$ be a constant such that for every $m \in \mathbb{N}$,  $\inf_{r \in [0,1)}\|\lambda_G((\mu_{m})_{r,\lazy})\|_2$ $<\lambda$. Set $\kappa_0=\sup_{m \in \N} \kappa_{S_m}$.
Choose $n_0$ large enough so that
\begin{equation}\label{eq.cont0}
\left|\frac{1}{n_0}\mathbb{E}_{\mu_\infty}[\kappa(L_{n_0})]-\ell(\mu_\infty)\right|<t_0,
\end{equation}
and 
\begin{equation}\label{eq.cont1}
2 \exp \left(\frac{-n_0 t_0^2 }{\kappa_0^2 32 \left(16\ln^+(\kappa_0)+ 8A_0/3 +33\right)^2(1-\sqrt{\lambda})^4} \right)<\frac{t_0}{\kappa_0},
\end{equation}
where the constant $A_0$ (depending only on $M$) is as in Remark \ref{rk.D}.

The choice of $n_0$ satisfying $\eqref{eq.cont1}$ implies by using Theorem \ref{thm.main.with.lambda} and the bound on the function $D$ given in Remark \ref{rk.D} which is non-decreasing in $\kappa$ and in $\lambda$, that for every $m \in \N$ large enough, we have
\begin{equation*}
\mathbb{P}_{\mu_m}\left(|\kappa(L_{n_0})-n_0\ell(\mu_m)| \geq n_0 t_0 \right) \leq \frac{t_0}{\kappa_0}.
\end{equation*}
Since  $  |\frac{1}{n_0}\kappa(L_{n_0}) - \ell(\mu_m)|\leq \kappa_0$,  this implies that for every $m$ large enough, we have
\begin{equation}\label{eq.cont2}
    \left|\frac{1}{n_0} \mathbb{E}_{\mu_m}[\kappa(L_{n_0})] -\ell(\mu_m)\right|\leq 2t_0.
\end{equation}

On the other hand, since $\mu_m \to \mu_\infty$ weakly, we have that as $m \to \infty$, $\mathbb{E}_{\mu_m}[\kappa(L_{n_0})] \to \mathbb{E}_{\mu_\infty}[\kappa(L_{n_0})]$. Therefore,  combining \eqref{eq.cont0} with \eqref{eq.cont2}, it follows that for every $m \in \N$ large enough, we have $|\ell(\mu_\infty)-\ell(\mu_m)|\leq 4t_0$ completing the proof.
\end{proof}

\begin{remark}A particular situation where the hypotheses of the previous result are satisfied is when there exists a finite set $S \subset \Isom(M)$ that contains the supports of all $\mu_m$ for $m \in \mathbb{N}$, $\mu_m \to \mu_\infty$ weakly and $\rho(\lambda_G(\mu_\infty))<1$.
This claim can easily be deduced from the results of Berg--Christensen \cite{berg-christensen}. We will omit the details as we will now prove a general continuity statement.
\end{remark}


The following result is the one can deduce from the unique cocycle-average property similarly to Hennion \cite{hennion} and Furstenberg--Kifer \cite{furstenberg-kifer}. For a very similar proof closer to our setting and related remarks, see Gou\"{e}zel--Math\'{e}us--Maucourant \cite[Proposition 2.3]{gouezel-matheus-maucourant}.
In the following, for a probability measure $\mu$ on $\Isom(M)$, we denote $L_1(\mu)=\int \kappa(g) d\mu(g)$.
 
\begin{proposition}\label{prop.continuity}
Let $(M,d)$ be a proper geodesic hyperbolic metric space. Let $\mu_n$ be a sequence of  non-elementary probability measures that converges weakly to a non-elementary probability measure $\mu$. Suppose furthermore that $L_1(\mu_n) \to L_1(\mu)$ as $n \to \infty$. Then,
$$
\ell(\mu_n) \to \ell(\mu).
$$
\end{proposition}

\begin{proof}
Let $\nu_n$ be a $\mu_n$-stationary probability measure on the horofunction boundary $X$ of $M$. By unique cocycle-average property \cite[Proposition 3.3(c)]{BQ.CLT.hyperbolic}, we have 
\begin{equation}\label{eq.cont1.prop}
\ell(\mu_n)=\int_{G \times X} \sigma(g,\xi) d\mu_n(g) d\nu_n(\xi).
\end{equation}
Since $X$ is compact, up to passing to a subsequence of $\nu_n$, we can suppose that the sequence $\nu_n$ converges to a probability measure $\nu$ on $X$. Since $\mu_n \to \mu$ weakly, one deduces from the continuity of the action of $G$ on $X$ that $\nu$ is $\mu$-stationary. Using the hypothesis that $L_1(\mu_n) \to L_1(\mu)$ and the fact that $\kappa(g) \geq |\sigma(g,\xi)|$ for every $g \in G$ and $\xi \in X$, one gets by dominated convergence that the sequence of integrals in \eqref{eq.cont1.prop} converges to $\int_{G \times X} \sigma(g,\xi) d\mu(g) d\nu(\xi)$. But by unique cocycle-average property, the latter is equal to $\ell(\mu)$. This implies the claimed convergence.
\end{proof}

Finally, we mention that, in the case of a countable hyperbolic group, further regularity properties of the drift are known, see e.g.\ Erschler--Kaimanovich \cite{erschler-kaimanovich}, Ledrappier \cite{Ledrappier.cts}, Gou\"{e}zel \cite{gouezel.cts} and Mathieu--Sisto \cite{mathieu-sisto}.

\section{The case of Gromov hyperbolic groups and rank-one linear groups}\label{sec.koubi.manu}


The goal of this section is to prove Corollaries \ref{corol.Koubi} and \ref{corol.manu} using Theorem \ref{thm.main.text}.

An important ingredient that allows us to obtain concentration inequalities with implied constants that depends, in a minimal fashion, on the probability measure $\mu$ is a version of uniform Tits alternative for group of isometries of hyperbolic spaces.
For hyperbolic groups, we will use Koubi's results \cite{koubi} and for linear groups the strong Tits alternative of Breuillard \cite{breuillard.strong.tits}.

\subsection{Concentration inequalities for random walks on Gromov hyperbolic groups}
 
For the proof of Corollary \ref{corol.Koubi}, we will use the following result of Koubi:

\begin{theorem}[\cite{koubi}]\label{thm.cite.koubi}
Let $\Gamma$ be a finitely generated non-elementary hyperbolic group. There exists $N_{\Gamma} \in \mathbb{N}$ such that for any finite subset $S$ generating $\Gamma$, there exists two elements $a,b \in S^{N_\Gamma}$ that generate a free subgroup of rank two. 
\end{theorem}
Here, by $S$-length of an element $g \in \Gamma$, we mean the distance of $g$ to the identity element in the word-metric induced by $S$.

The previous result will be useful to us in combination with the following straightforward observation (see e.g.\ \cite[\S 8]{breuillard.sl2}).

\begin{lemma}\label{lemma.uniformtits.radius}
Let $\Gamma$ be a countable group and $S \subset \Gamma$ such that $S^{N_0}$ contains a pair of elements that generates a free subgroup of rank two for some $N_0 \in \N$. 
Let $\mu$ be a probability measure with support $S\cup \{\id\}$ and set $m_\mu=\min_{g \in S}\mu(g)$. Then,
 $$\|\lambda_\Gamma(\mu)\|_2 \leq    \left(1- \left(1-\frac{\sqrt{3}}{2}\right)  m_{\mu}^{2{N_0}} \right)^{\frac{1}{2{N_0}}}
 .$$

\end{lemma}

\begin{proof}
Consider the probability measure $\mu'=\mu \ast \check{\mu}$ and denote by $S'$ its support. Since $S$ contains identity, the set $S'$ is symmetric and it contains $S$. It follows that $(S')^{N_0}$ contains a set $\{a,b,a^{-1},b^{-1}\}$, where $a,b$ are the generators of a free group of rank two.

Since $\mu'$ is symmetric, the operator $\lambda_\Gamma(\mu')$ on $\ell^2(\Gamma)$ is self-adjoint, and since $\lambda_\Gamma(\mu'^{\ast N_0})=\lambda_\Gamma(\mu')^{N_0}$, we have $\|\lambda_\Gamma(\mu')\|_2=
\|\lambda_\Gamma(\mu'^{\ast N_0})\|_2^{1/N_0}$. 
Therefore, 
\begin{equation}
 \|\lambda_\Gamma(\mu)\|_2=\|\lambda_\Gamma(\mu')\|_2^{\frac{1}{2}}=\|\lambda_\Gamma(\mu'^{N_0})\|_2^{\frac{1}{2N_0}}\label{eq.a1}\end{equation}
 
On the other hand,  we write 
$\mu'^{ \ast N_0}=m_{\mu'}^{N_0} \eta+ (1-m_{\mu'}^{N_0})\zeta$, where $\eta$ is the uniform probability measure on $\{a,b,a^{-1}, b^{-1}\}$  and $\zeta$ some probability measure on $\Gamma$. Using the trivial bound $\|\lambda_\Gamma(\zeta)\|_2 \leq 1$, we deduce that 
\begin{equation}\label{eq.a2}
\|\lambda_\Gamma(\mu'^{\ast {N_0}})\|_2\leq  1-\kappa \,m_{\mu'}^{N_0},
\end{equation}
where $1-\kappa=\sqrt{3}/2$ is the  spectral radius  of the uniform probability measure on the free group \cite[Theorem 3]{kesten.symmetric}. 
Combining \eqref{eq.a1} and \eqref{eq.a2}, and using the fact that $m_{\mu'}\geq m_{\mu}^2$, we deduce that 
\begin{equation}\|\lambda_\Gamma(\mu)\|_2 \leq \left(1- \kappa \,m_{\mu}^{2{N_0}} \right)^{\frac{1}{2{N_0}}}. 
\end{equation}
\end{proof}

\begin{proof}[Proof of Corollary  \ref{corol.Koubi}]
Note first that by \cite[Proposition 2.6]{CCMT}, the group $\Gamma$ is a non-elementary hyperbolic group. 
Therefore, the hypothesis of Lemma \ref{lemma.uniformtits.radius} is satisfied for every finite generating set of $\Gamma$ with a uniform constant $N_0=N'_\Gamma$ thanks to Koubi's Theorem
\ref{thm.cite.koubi}. Applying  Lemma  \ref{lemma.uniformtits.radius}  to $\mu_{1/2,\lazy}$ and using the fact that 
$m_{\mu_{1/2,\lazy}} \geq
\frac{1}{2} m_{\mu}$ yields that 
\begin{equation}\label{eq.explicit.bound.koubi}
\|\lambda_\Gamma(\mu_{1/2, \lazy})\|_2\leq  \left( 1- \frac{(1-\sqrt{3}/2)}{2^{N'_{\Gamma}}}  \,m_{\mu}^{N'_\Gamma}\right)^{\frac{1}{2 N'_{\Gamma}}} \leq  \left( 1- \frac{(1-\sqrt{3}/2)}{N'_{\Gamma}2^{N'_{\Gamma}+1}}  \,m_{\mu}^{N'_\Gamma}\right).
\end{equation}

Observe finally that by the discreteness assumption 
of $\Gamma$ and by Berg-Christensen's Corollary   \cite[Corollaire 3]{berg-christensen}, one has that 
$\|\lambda_\Gamma(\mu_{1/2, \lazy})\|_2=\|\lambda_G(\mu_{1/2, \lazy})\|_2$. Using now 
Theorem \ref{thm.main.with.lambda} and the expression of the function $D(\cdot, \cdot)$ given in Remark \ref{rk.D}, we get the desired result with 

$$\alpha_{\Gamma}:=2^{21+4N'_{\Gamma}} \frac{(N'_\Gamma)^4}{(1-\sqrt{3}/2)^4},$$
$A_M=A_0+3$ and $N_\Gamma=4N'_\Gamma$.
\end{proof}

\subsection{Concentration inequalities for random walks on rank-one semisimple linear groups}

For the proof of Corollary \ref{corol.manu}, we will use the following result of Breuillard \cite[Theorem 1.1]{breuillard.strong.tits} and \cite{breuillard.height} (see \cite{breuillard.sl2} for the particular case of $\SL_2$).

\begin{theorem}[\cite{breuillard.strong.tits}]\label{thm.cite.manu}
For every $d \in \N$ there is $N_d \in \N$ such that if $\mathrm{k}$ is any field and $S$ is a finite symmetric subset of $\GL_d(\mathrm{k})$ containing identity, either $S^{N_d}$ contains two elements which generate a non-abelian free group, or the group generated by $S$ contains a finite-index solvable subgroup.
\end{theorem}

We are now able to give the

\begin{proof}[Proof of Corollary \ref{corol.manu}]
Let $d \in \N$ be given, and $\kk$ and $\mathbb{H}\subseteq \mathbb{SL}_d$ be as in the statement. Let the natural number $N'_d$ (depending only on $d$) be as given by Theorem \ref{thm.cite.manu}. Let $\mu$ be a probability measure whose support $S$ is a finite subset of $\mathbb{H}(\mathrm{k})$ that generates a discrete non-amenable subgroup $\Gamma$ of $\mathbb{H}(k)$. 
Let $\mu':=\mu_{1/2, \lazy}$, $\mu'':=\mu'\ast \mu'^{-1}$, denote by $S'$ the support of $\mu'$ and $S''$ the support of $\mu''$. Notice that the finite set $S''$ is symmetric, contains the identity and it generates the group $\Gamma$. Therefore it follows by Theorem \ref{thm.cite.manu} that $S''^{N'_d}$ contains two elements that generate a free group of rank two where the constant $N'_d \in \N$ only depends on the dimension $d$. Applying Lemma \ref{lemma.uniformtits.radius}, we get that 
$$\|\lambda_\Gamma(\mu'')\|_2 \leq    \left(1- (1-\sqrt{3}/2) m_{\mu''}^{2{N'_d}} \right)^{\frac{1}{2{N'_d}}}.$$
Since $\|\lambda_\Gamma(\mu'')\|_2=\|\lambda_\Gamma(\mu')\|_2^2$ and $m_{\mu''}\geq m_{\mu'}^2\geq m_{\mu}^2/4$, we deduce
\begin{equation}\label{eq.explicit.bound.manu}
\|\lambda_\Gamma(\mu')\|_2 \leq 1-m_{\mu}^{4N'_d} \frac{1-\sqrt{3}/2}{4N'_d 2^{4N'_d}}.
\end{equation}
As in the proof of Corollary \ref{corol.Koubi}, by the discreteness assumption 
on $\Gamma$ it follows that 
$\|\lambda_\Gamma(\mu_{1/2, \lazy})\|_2=\|\lambda_G(\mu_{1/2, \lazy})\|_2$. Therefore,
a direct application of Theorem   \ref{thm.main.with.lambda} (with $r=1/2$ on the right hand side of the theorem) and   the expression of the function $D(\cdot, \cdot)$ given in Remark \ref{rk.D} concludes the proof with 
$N_d=16 N_d'$, $\alpha_d=2^{25+N_d} N_d^4 \frac{1}{(1-\sqrt{3}/2)^4}$ and $A_{\mathbb{H}, \kk}=A_0/3+3$, where $A_0$ is the constant  defined in Remark \ref{rk.explicit.C} applied for the isometry group $G$ of the symmetric space $M$ associated to the rank-one group $\mathbb{H}(\kk)$. 
\end{proof}

\section{Probabilistic free subgroup theorem}\label{sec.tits}

The goal of this section is to prove Theorem \ref{thm.proba.tits} from Introduction. To do this, we start by proving a general result which shows that uniform large deviation estimates for the Busemann cocycle together with positivity of the drift imply a probabilistic free subgroup theorem for isometries of Gromov hyperbolic spaces.

\subsection{Free subgroups from uniform large deviations}\label{subsec.ULD}

Let $(M,d)$ be a $\delta$-hyperbolic metric space and fix $o \in M$. Let $\mu$ be a Borel probability measure on $\Isom(M)$ endowed with the topology of pointwise convergence. 

We introduce the following uniform large deviation hypothesis for a probability measure $\mu$ with finite first order moment on $\Isom(M)$:\\[-5pt]

\textbf{ULD:} For every $\epsilon>0$ and $n \in \N$, there exist non-negative constants $p_n(\epsilon)$ such that for every $\epsilon>0$, $p_n(\epsilon) \to 0$ as $n \to \infty$ and 
\begin{equation}\label{eq.ULD}
\sup_{y \in M}\mathbb{P}(|\sigma(R_n^{\pm 1},y)-n\ell(\mu)| \geq n\epsilon) \leq p_n(\epsilon),
\end{equation}
where $\sigma$ denotes the Busemann cocycle. Note that, whenever the \textbf{ULD} hypothesis is satisfied, by replacing, for every $\epsilon>0$, $p_n(\epsilon)$ by $\sup_{m \geq n}p_n(\epsilon)$, we can and we will suppose that it is satisfied with a non-increasing sequence $p_n(\epsilon)$.

The rest of \S \ref{subsec.ULD} is devoted to the proof of the following

\begin{proposition}\label{prop.ULD.tits}
Let $(M,d)$ be a $\delta$-hyperbolic metric space and $\mu$ a probability measure on $\Isom(M)$ with finite first order moment. Suppose that $\mu$ satisfies the hypothesis \textbf{ULD} and $\ell(\mu)>0$. Then, for every integer $n > 2+ \frac{16\delta}{\ell(\mu)}  $, we have
$$(\mu^{\ast n}\otimes \mu^{\ast n} )\left\{(\gamma_1, \gamma_2) : \langle \gamma_1, \gamma_2 \rangle \, \textrm{is free} \right\}> 1-25p_{\lfloor n/2 \rfloor}(\ell(\mu)/8).$$
\end{proposition}

Before proceeding with the proof, we make a few remarks on its hypotheses.


\begin{remark}[About \textbf{ULD} hypothesis]

1. Theorem \ref{thm.main.text} shows that  the \textbf{ULD} hypothesis is verified, with explicit constants,  for random walks on a proper hyperbolic space $M$ such that $\Isom(M)$ acts cocompactly on $M$. This explicit aspect will be crucial for the quantitative probabilistic free subgroup Theorem \ref{thm.proba.tits}.\\[3pt]
2. However, \textbf{ULD} (with qualitative constants) is satisfied also when $M$ is not proper:
using cocycle large deviation results of \cite{BQ.CLT.linear}, it can be shown (see \cite[Proposition 2.8]{horbez}) that if $M$ is a separable geodesic   hyperbolic space, then \textbf{ULD} hypothesis holds for any countably supported non-elementary probability measure $\mu$ with finite second order moment. Moreover, in this case,   $\ell(\mu)>0$ (\cite[Theorem 1.2]{maher-tiozzo})
\end{remark}

We will show that with high probability, two independent random walks $R_n$ and $R'_n$ will play ping-pong on the space $M$. To set the random ping-pong table, we need some geometric lemmas. Let $(M,d)$ be a $\delta$-hyperbolic space,  fix $o\in M$ and let $C>0$. Recall that the shadow of $y\in M$ seen from $x\in X$ is the following subset of $M$; 
$$\mathcal{O}_C(x,y)=\{z\in M : (z|y)_x\geq d(x,y)-C\}.$$
It is immediate that \begin{equation}\mathcal{O}_C(x,y)=\{z\in M : (z|x)_y\leq C\}.\label{shadow2}\end{equation}
Observe that $\mathcal{O}_C(x,y)=M$ when $C\geq d(x,y)$. 
 
We will use the following lemma to construct Schottky subgroups of $\Isom(M)$.

\begin{lemma}\label{lemma.ping-pong} Let $(M,d)$ be a $\delta$-hyperbolic space and $o\in M$. Let $\gamma_1, \gamma_2\in \Isom(M)$. Suppose that there exists $D>0$ such that
\begin{enumerate}
\item  for every $i \neq j \in\{1, 2\}$ and $\epsilon_{1},\epsilon_2 \in \{-1, 1\}$, 
$(\gamma_i^{\epsilon_1} \cdot o | \gamma_j^{\epsilon_2}\cdot o)_o\leq D$. 
\item  for every $i=\{1,2\}$,  $(\gamma_i \cdot o | \gamma_i^{-1} \cdot o)_o\leq D$,
\item $0<\frac{1}{2}\max\{\kappa(\gamma_1), \kappa(\gamma_2)\}<\min\{\kappa(\gamma_1), \kappa(\gamma_2)\} -D-\delta$. 
\end{enumerate}
Then $\langle \gamma_1, \gamma_2\rangle$ is non-abelian free group. 
\end{lemma}
 
\begin{remark}
Assumptions (ii) and (iii) have the consequence that $\gamma_1$ and $\gamma_2$ are both hyperbolic isometries. Indeed it follows from (iii) that we have $\kappa(\gamma_i) >  2D+2\delta$ for $i=1,2$. Together with (ii), this implies $\kappa(\gamma_i^2)>\kappa(\gamma_i)+2\delta$, which in turn imply that $\gamma_i$ is a hyperbolic isometry by \cite[Lemma 2.2]{hyperbolic-book}.
\end{remark}
The proof of Lemma \ref{lemma.ping-pong} will follow from intermediate lemmas inspired from \cite[Appendix A]{BMSS}. 
\begin{lemma}\label{lemma.image.of.shadow} Let $(M,d)$ be any metric space and $\gamma\in \Isom(M)$. Then for every constant $C\geq 0$, we have 
$$\gamma (M\setminus \mathcal{O}_C(o, \gamma^{-1} \cdot o)) \subseteq \mathcal{O}_{\kappa(\gamma)-C}(o,\gamma \cdot o).$$ 
\end{lemma}
\begin{proof}
We have  clearly for every $C\geq 0$, 
$$\gamma \cdot (M\setminus \mathcal{O}_C(o, \gamma^{-1}\cdot o))=M\setminus \mathcal{O}_C(\gamma \cdot o, o ).$$
So let $x\not\in  \mathcal{O}_C(\gamma \cdot o, o)$. By \eqref{shadow2} this means that $(x | \gamma \cdot o)_o>C$.
Using the identity
$$\kappa(\gamma)=(x|o)_{\gamma \cdot o}+(\gamma \cdot o | x)_{o}, $$
we deduce that $(x|o)_{\gamma \cdot o}<\kappa(\gamma)-C$. Hence by \eqref{shadow2}, we have $x\in \mathcal{O}_{\kappa(\gamma)-C}(o, \gamma \cdot o)$ as desired.  
\end{proof}

\begin{lemma}\label{lemma.disjoint} Let $(M,d)$ be a $\delta$-hyperbolic space. 
Let $\gamma_1, \gamma_2 \in Isom(M)$, $D>0$. Denote $\kappa_{1,2}:=\min\{\kappa(\gamma_1), \kappa(\gamma_2)\}$. 
$$\begin{rcases}
(\gamma_1 \cdot o | \gamma_2 \cdot o)_o \leq D\\
    \kappa_{1,2}> D+\delta \end{rcases} 
     \Longrightarrow \forall 0<C<\kappa_{1,2}-D-\delta,\, \mathcal{O}_C(o, \gamma_1 \cdot o) \cap \mathcal{O}_C(o, \gamma_2 \cdot o)=\emptyset.$$
 \end{lemma}
 
 \begin{proof}
 Assume that $(\gamma_1 \cdot o, \gamma_2 \cdot o)_o \leq D$ and $\kappa_{1,2}\geq D+C+\delta$. 
Let $x\in \mathcal{O}_C(o, \gamma_1\cdot o)$. By definition, $(x,\gamma_1\cdot o)_o\geq \kappa(\gamma_1)-C> D+\delta $. But by  $\delta$-hyperbolicity, 
$$\min\{(x | \gamma_2\cdot o)_o, (x | \gamma_1\cdot o)_o\}\leq (\gamma_1 \cdot o |  \gamma_2 \cdot o)_o+\delta\leq D+\delta.$$ Thus $(x | \gamma_2 \cdot o)_o\leq D+\delta<\kappa(\gamma_2)-C$ and hence $x\not\in \mathcal{O}_C(o, \gamma_2\cdot o)$ which proves the claim.
\end{proof}
 
Now we are able give 
 
\begin{proof}[Proof of Lemma \ref{lemma.ping-pong}]
Using the assumption (iii), fix any real $C$ such that 
$$\frac{1}{2}\max_{i=1,2}\{\kappa(\gamma_i)\} < C< \min_{i=1,2}\{\kappa(\gamma_i)\}-D-\delta.$$  For $i=1,2$, denote   $\mathcal{O}_i=\mathcal{O}_C(o, \gamma_i \cdot o)$ and $\mathcal{O}_i^{<}=\mathcal{O}_C(o, \gamma_i^{-1} \cdot o)$. By Lemma \ref{lemma.disjoint}, these are four disjoint subsets of $M$. Moreover, 
by Lemma \ref{lemma.image.of.shadow} and the choice of the constant $C$, the following inclusions  hold  every $i=1,2$, 
$$\gamma_i (M\setminus \mathcal{O}_i^{<}) \subseteq \mathcal{O}_i$$ and 
$$\gamma_i^{-1} (M\setminus \mathcal{O}_i) \subseteq \mathcal{O}_i^<.$$ Thus pair of elements
$\gamma_1, \gamma_2$ satisfies the hypotheses of the classical ping-pong lemma and therefore 
they generate then a free subgroup of $\Isom(M)$. 
\end{proof}

With Lemma \ref{lemma.ping-pong} at hand, we focus now on showing that the random walks $R_n, R'_n, R_n^{-1}, {R'_n}^{-1}$ satisfy assumptions (i)--(iii) of Lemma \ref{lemma.ping-pong} with $D=n\ell(\mu)/8+2\delta$, with probability tending to one depending on the constants $p_n(\epsilon)$ appearing in the hypothesis \textbf{ULD}. Before that, we  provide some estimates on the random walk $R_n$ based on uniform large deviation estimates.

\begin{lemma}\label{lemma.estimate.rw}
Let $(M,d)$ be a $\delta$-hyperbolic metric space and let $\mu$ be a probability measure on $\Isom(M)$ with finite first order moment and satisfying the hypothesis \textbf{ULD}. Then, the following estimates hold. 
\begin{enumerate}
\item For every $\epsilon>0$ and every $n\in \N$,  
$$\sup_{y\in M}
    \p\left( (R_n \cdot o| y)_o\geq \epsilon n\right) \leq  2p_n(\epsilon).$$
\item    For every $0<\epsilon\leq \ell(\mu)/8$ and every $n> 2+\frac{8 \delta}{\ell(\mu)}$,  $$
    \p\left( (R_n \cdot o| R_n^{-1}\cdot o)_o\geq \epsilon n + 2\delta\right)  \leq 8p_{\lfloor n/2 \rfloor}(\epsilon) .$$
\end{enumerate}
\end{lemma}

\begin{proof}

(i)   
Using the identity  $$(go| y)_o = \frac{1}{2} (\kappa(g)-\sigma(g^{-1},y))$$
which holds for any $g\in \Isom(M)$ and $y\in M$, the desired inequality follows from \textbf{ULD} hypothesis applied to both $\kappa(R_n)=\sigma(R_n,o)$ and $\sigma(R_n^{-1},y)$.

(ii) Let $\epsilon>0$ and $n\in \N$. 
For every $1\leq m< n$, we denote 
$R_{m,n}:=X_m \cdots X_n$. 
By $\delta$-hyperbolicity, we have  
\begin{equation}\label{hyperbolicity.relation2}
\begin{aligned}
\min &\left\{(R_n\cdot o | R_n^{-1} \cdot o)_o, (R_n \cdot o| R_{\lfloor n/2 \rfloor} \cdot o)_o, (R_n^{-1} \cdot o| R_{\lfloor n/2 \rfloor +1,n}^{-1} \cdot o)_o\right\}\\ &\leq (R_{\lfloor n/2 \rfloor}\cdot o | R_{\lfloor n/2 \rfloor+1,n}^{-1} \cdot o)_o+2\delta.
\end{aligned}
\end{equation}
On the one hand,  since $R_{\lfloor n/2 \rfloor}=X_1 \cdots X_{\lfloor n/2 \rfloor}$ and 
$R_{\lfloor n/2 \rfloor+1,n}=X_{\lfloor n/2 \rfloor+1} \cdots X_{n}$ are independent random variables, we deduce from (i) that 
\begin{equation}\label{b11}\p\left((R_{\lfloor n/2 \rfloor}\cdot o | R_{\lfloor n/2 \rfloor+1,n}^{-1} \cdot o)_o\geq \epsilon n\right)\leq 2p_{\lfloor n/2\rfloor}(\epsilon).
\end{equation}
On the other hand, we claim that if $0<\epsilon< \ell(\mu)/8$ and $n>8\delta/\ell(\mu)$, then the following holds: 
\begin{equation}\label{max}\p \left(\min \left\{ (R_n^{-1}\cdot o | R_{\lfloor n/2 \rfloor+1,n}^{-1}\cdot  o)_o, (R_n\cdot o | R_{\lfloor n/2 \rfloor}\cdot  o)_o \right\}\leq \epsilon n+ 2\delta\right) \leq 2p_n(\epsilon)+4p_{\lfloor n/2 \rfloor}(\epsilon).\end{equation}
This will finish the proof of (ii) by combining \eqref{hyperbolicity.relation2}, \eqref{b11} and \eqref{max}. We now check \eqref{max}. 
We have that  
\begin{equation}\label{interm}(R_n\cdot o | R_{\lfloor n/2 \rfloor} \cdot o)_o=\frac{\kappa(R_n)+\kappa(R_{\lfloor n/2 \rfloor}) - \kappa(R_{\lfloor n/2 \rfloor+1,n})}{2}.\end{equation}
Thanks to \textbf{ULD}, the following inequalities hold: $\p\left(\kappa(R_n)<n(\ell(\mu)-\epsilon)\right)\leq p_n(\epsilon)$ and $\p\left(\kappa(R_{\lfloor n/2 \rfloor})< \lfloor n/2 \rfloor(\ell(\mu)-\epsilon)\right)\leq p_{\lfloor n/2 \rfloor}(\epsilon)$. Moreover, since the $X_i$'s are iid, for each $n \in \N$, the distribution of $\kappa(R_{\lfloor n/2 \rfloor+1, n})$ is the same as $\kappa(R_{n-\lfloor n/2 \rfloor})$. Thus by applying again the \textbf{ULD} hypothesis, we get that $$\p\left(\kappa(R_{\lfloor n/2 \rfloor+1,n})> (n-\lfloor n/2 \rfloor)(\ell(\mu)+\epsilon)\right)\leq
p_{n-\lfloor n/2 \rfloor}(\epsilon)\leq 
p_{\lfloor n/2 \rfloor}(\epsilon).$$
 By \eqref{interm} this yields that   
\begin{equation}\label{b12}\p\left((R_n\cdot o |  R_{\lfloor n/2 \rfloor}\cdot  o)_o\leq  \lfloor n/2 \rfloor\ell(\mu)-n\epsilon\right)\leq p_n(\epsilon)+2p_{\lfloor n/2 \rfloor}(\epsilon).\end{equation}
A similar relation holds by replacing the couple $(R_n\cdot o | R_{\lfloor n/2 \rfloor}\cdot o)_o$ with the couple  $(R_n^{-1} \cdot o| R_{\lfloor n/2 \rfloor +1,n}^{-1} \cdot o)_o$.  Consequently 
estimate \eqref{max} holds  as soon as 
$\lfloor n/2\rfloor \ell(\mu) -n\epsilon > \epsilon n+2\delta$. This is for instance guaranteed if   $0<\epsilon\leq  \ell(\mu)/8$ and $n>2+8\delta/\ell(\mu)$. This shows  \eqref{max}. Since $p_n(\epsilon)$ is non-increasing, this concludes the proof of estimate (ii). 
 \end{proof}

We are finally ready to conclude 
 
\begin{proof}[Proof of Proposition \ref{prop.ULD.tits}] Consider two independent random walks $(R_n)_{n\geq 1}$ and $ (R'_n)_{n\geq 1}$ driven by $\mu$. 
We will check that  $R_n, R'_n, R_n^{-1}, {R'_n}^{-1}$ satisfy assumptions (i)--(iii) of Lemma \ref{lemma.ping-pong} with $D_n:=n\ell(\mu)/8+2\delta$, with probability tending to one. By (i) of Lemma \ref{lemma.estimate.rw} and the independence of the random variables $R_n$ and $R'_n$ we deduce that 
\begin{equation}\label{u1}\p\left( (R_n\cdot o | {R'_n} \cdot o)_o \geq n\ell(\mu)/8 \right)\leq  2 p_n(\ell(\mu)/8).
\end{equation}
Three other similar estimates hold by replacing the couple $(R_n, R'_n)$ with the couples $(R_n, {R'_n}^{-1})$, $(R_n^{-1}, R'_n)$, $(R_n^{-1}, {R'_n}^{-1})$. 
Also, by (ii) of Lemma \ref{lemma.estimate.rw}, we have for $n>2+8\delta/\ell(\mu) $, 
\begin{equation}\label{u2}\p\left( (R_n \cdot o | R_n^{- 1} \cdot o)_o \geq D_n\right)\leq 8 p_{\lfloor n/2 \rfloor}(\ell(\mu)/8), \end{equation} and similarly 
\begin{equation}\label{u3}\p\left( (R'_n \cdot o| {R'_n}^{- 1} \cdot o)_o \geq D_n\right)\leq 8 p_{\lfloor n/2 \rfloor}(\ell(\mu)/8). \end{equation}
Finally, using the hypothesis \textbf{ULD}, we have for every $\epsilon>0$ and $n\in \N$  that $$\p \left( \kappa(R_n)\in [n \ell(\mu) - n\epsilon, n \ell(\mu)+n\epsilon]\right)\geq 1-p_n(\epsilon),$$
and similarly for $\kappa(R'_n)$.
Hence, with probability $\geq 1-p_n(\epsilon)$,
\begin{equation}\label{u4}0<\frac{1}{2}\max\{\kappa(R_n), \kappa(R'_n)\}<\min\{\kappa(R_n), \kappa(R'_n)\} -D_n-\delta\end{equation}
as soon as $$ (n \ell(\mu)+n\epsilon)/2< n\ell(\mu)-n\epsilon -D_n-\delta,$$
and in particular as soon as $0<\epsilon\leq  \ell(\mu)/8$ and $n>16\delta/\ell(\mu)$.

Finally, specializing to $\epsilon=\ell(\mu)/8$,  we conclude that the seven estimates \eqref{u4}, \eqref{u3}, \eqref{u2} \eqref{u1}, and the three other inequalities similar to \eqref{u1} hold  simultaneously in an event of $\p$-probability 
\begin{equation}\label{eq.tits.end}
>1-\left(8p_n(\ell(\mu)/8)+16p_{\lfloor n/2 \rfloor}(\ell(\mu)/8) + p_n(\ell(\mu)/8)\right)
>1-25 p_{\lfloor n/2 \rfloor}(\ell(\mu)/8),
\end{equation}
provided that $n>2+16\delta/\ell(\mu)$. In other words, for every such $n \in \N$, with probability at least the amount given by \eqref{eq.tits.end}, the elements $R_n$ and $R_n'$ satisfy the hypotheses of Lemma \ref{lemma.ping-pong} and this finishes the proof of Proposition \ref{prop.ULD.tits}. 
\end{proof}

\subsection{A lower bound for the drift}\label{subsec.lower.bound}
In view of Theorem \ref{thm.main.text} and Proposition \ref{prop.ULD.tits}, the only remaining ingredient for the proof of Theorem \ref{thm.proba.tits} is a control of how small the drift $\ell(\mu)$ of the random walk can be. The harmonic analytic approach of \S \ref{sec.explicit} allows one to deduce a lower bound on the drift as we now discuss. This sort of result should be known to the experts. Results of similar flavor appear in the works \cite{guivarch.sur.la.loi,  margulis.book,nevo,virtser}.

Given $R \geq 0$, as before, we set $B_R=\{g \in G \; | \; d(go,o) \leq R\}$. Since $M$ is proper, the sets $B_R$ defined above are compact and they have non-empty interior if $R>0$. In particular, there exists $K_0 \in \N$ and $g_1, \ldots, g_{K_0} \in B_{6D_1}$ such that $B_{6D_1} \subseteq \cup_{i=1}^{K_0}g_i B_{D_1}$, where, as before, $D_1 \in \R$ denotes the constant $\max \{D_0,1\}$ and $D_0:=2\textrm{diam}(G\backslash M)$. For convenience later on, we choose $K_0$ to be the smallest such integer. An elementary covering argument allows one to get the bound $K_0\leq \frac{\mu_G(B_{\frac{13}{2}D_1 })}{\mu_G(B_{\frac{1}{2}D_1})}$.


\begin{proposition}\label{prop.drift.lower.bound}
Let $(M,d)$ be a proper geodesic metric space such that $G=\Isom(M)$ acts cocompactly on $M$. Then, for every probability measure $\mu$ on $G$ with finite first order moment, the drift $\ell(\mu) \in \R$ satisfies
\begin{equation}\label{eq.lyap.lower}
\ell(\mu) \geq \frac{2D_1}{\ln K_0} \sup_{r \in [0,1)} \frac{1}{1-r} \ln \frac{1}{\|\lambda_G(\mu_{r,\lazy})\|_2}.
\end{equation}
\end{proposition}

\begin{remark}\label{rk.lazy.again}
The reason why we also include $\mu_{r,\lazy}$ in the conclusion of the previous Proposition \ref{prop.drift.lower.bound} is that, as discussed in Remark \ref{rk.lazy.discussion}, when $\mu$ is non-symmetric, it might happen that the closed group $\overline{\Gamma}_\mu$ generated by the support of $\mu$ is non-amenable whereas $\|\lambda_G(\mu)\|_2=1$. However, in this case, for every $r>0$, we have $\|\lambda_G(\mu_{r,\lazy})\|_2<1$. Therefore, whenever $\overline{\Gamma}_\mu$ is non-amenable the lower bound provided by the proposition is  strictly positive and it depends only on $D_1,K_0$, and  $\|\lambda_G(\mu_{1/2,\lazy})\|$.
\end{remark}

\begin{proof}
We first prove that  
$$
\ell(\mu) \geq\frac{2D_1}{\ln K_0} \ln \frac{1}{\|\lambda_G(\mu)\|_2}.
$$
The proposition  then follows by applying the above for  each  $\mu_{r,\lazy}$ and noting that $\ell(\mu_{r,\lazy})=(1-r)\ell(\mu)$.
A straightforward modification of the proof of Lemma \ref{lemma.bq.5.2} shows that for every $R>0$ and $n \in \N$, we have
\begin{equation}\label{eq.bq.5.2}
\mathbb{P}(d(R_n o,o) \leq R) \leq \left(\frac{\mu_G(B_{2R})}{\mu_G(B_{R})}\right)^{1/2} \|\lambda_G(\mu)\|_2^n,
\end{equation}
where $\mu_G$ is a Haar measure on $G$.
Indeed, the additional term $D_0$ in the left hand side of \eqref{estime.1} disappears since here we take $m=m'=o$.

We now claim that for every $r \geq D_1$,
\begin{equation}\label{eq.macornu}
    \frac{\mu_G(B_{r+D_1})}{\mu_G(B_{r})} \leq K_0,
\end{equation}
where $K_0 \in \N$ is the constant defined before the statement of Proposition \ref{prop.drift.lower.bound}.
Indeed, given $r \geq D_1$, let $\{\gamma_1,\ldots,\gamma_T\}$ be a maximal $2D_1$-separated set contained in $B_{r-D_1}$ with respect to the left-invariant pseudo-metric $d_G$ defined as $d_G(g,h)=d(go,ho)$ for every $g,h \in G$. 
%
Then the collection $\gamma_iB_{D_1}$ for $i=1,\ldots,T$ consists of disjoint compact subsets of $B_r$ of same Haar measure as $B_{D_1}$ so that we have $\mu_G(B_r) \geq T \mu_G(B_{D_1})$.  On the other hand, since $G$ acts co-compactly on $M$ and $M$ 
is geodesic, it is not hard to see that 
every element in $B_{r+D_1}$ is $2D_1$-close for the pseudo-metric $d_G$ to an element of $B_r$ (in fact, $(G,d_G)$ is a large-scale geodesic space in the sense of \cite[Definition 3.B.1]{cornulier-harpe}). Hence the collection $\gamma_i B_{6D_1}$ for $i=1,\ldots,T$ is a covering of $B_{r+D_1}$ by compacts having the same Haar measure as $B_{6D_1}$ and therefore we have $\mu_G(B_{r+D_1}) \leq T \mu_G(B_{6D_1})$. Therefore we deduce $\frac{\mu_G(B_{r+D_1})}{\mu_G(B_r)} \leq \frac{\mu_G(B_{6D_1})}{\mu_G(B_{D_1})} \leq K_0$ proving \eqref{eq.macornu}. 

Now, by using \eqref{eq.macornu} iteratively and plugging it in \eqref{eq.bq.5.2}, we deduce that for every $\alpha<\frac{2D_1}{\ln K_0}\ln \frac{1}{\|\lambda_G(\mu)\|_2}$, we have
\begin{equation}\label{stronger}
\limsup_{n \to \infty} \mathbb{P}(d(R_no,o) \leq \alpha n)^{\frac{1}{n}}<1.
\end{equation}
The result follows in view of the Kingman's subadditive ergodic theorem.
\end{proof}

\begin{remark}
In fact, the estimate \eqref{stronger} above provides a lower bound $L_0>0$ for a region of type $[0,L_0)$ on which the large deviation rate function of the process $(\frac{1}{n}\kappa(R_n))_{n \geq 1}$ is positive. Such a lower bound is, a priori, stronger than a lower bound for the drift $\ell(\mu)$. However, the recent works \cite{BMSS} under finite exponential moment and \cite{gouezel.first.moment} under finite first order moment assumptions, identify the drift $\ell(\mu)$ as the smallest real $r$ such that the rate function $I$ is positive on $[0,r)$. 

On the other hand, the fact that Proposition \ref{prop.drift.lower.bound} provides an explicit region of positivity of $I$ allows, for example, to obtain explicit constants in \cite[Theorem 1.2]{maher-tiozzo} under our assumptions.
\end{remark}

\subsection{Proof of Theorem \ref{thm.proba.tits}}
We denote by $D(.,.)$ the positive function given by Theorem \ref{thm.main.text}. The hypotheses of Theorem \ref{thm.proba.tits} allow us to apply Theorem \ref{thm.main.text} to deduce that for every $r \in [0,1)$ the probability measure $\mu$ satisfies the \textbf{ULD} hypothesis with 
\begin{equation}\label{eq.getuld1}
p_n(\epsilon)=2\exp\left(\frac{-n\epsilon^2}{\kappa_{S}^2D(\kappa_{S},\|\lambda_G(\mu_{r,\lazy})\|_2)}\right).
\end{equation}

Applying Proposition \ref{prop.ULD.tits}, we deduce that for every integer
\begin{equation}\label{eq.range}
n > 2+ \frac{16\delta}{\ell(\mu)},    
\end{equation}
we have
$$(\mu^{\ast n}\otimes \mu^{\ast n} )\left\{(\gamma_1, \gamma_2)\,|\, \langle \gamma_1, \gamma_2 \rangle \, \textrm{is free} \right\}> 1-25p_{\lfloor n/2 \rfloor }(\ell(\mu)/8)>1-25p_{n/4}(\ell(\mu)/8).$$
Therefore, using  \eqref{eq.getuld1} and  the bound provided by Proposition \ref{prop.drift.lower.bound}, setting $\lambda_r=\|\lambda_G(\mu_{r,\lazy})\|_2$ we obtain that for every $r\in [0,1)$ and for every $n> 2+\frac{8 \delta (1-r) \ln K_0 }{D_1 \ln \frac{1}{\lambda_r}}$,  two independent random walks generate a free subgroup with probability 
\begin{equation}\label{eq.the.real.tits.bound}
>
1-50\exp\left(\frac{-nD_1^2 (\ln \lambda_r)^2 (1-\sqrt{\lambda_r})^4}{2^{11} (\ln K_0)^2 (1-r)^2  \kappa_S^2 (16 \ln^+(\kappa_S)+\frac{8A_0}{3}+33)^2}\right).
\end{equation}
Specifying to $r=1/2$, the result follows by taking the function $n_0(\cdot)$
and $T(\cdot, \cdot)$ as
$$T(\kappa, \lambda)= \tilde{B}_M \frac{(\ln\lambda)^2 (1-\sqrt{\lambda})^4}{\kappa^2(\ln^+\kappa +\tilde{A}_M)^2},
$$
$$n_0(\lambda)=2-\tilde{C}_M \frac{1}{\ln \lambda},$$
where the constants $\tilde{A}_M,\tilde{B}_M,\tilde{C}_M$ are given by
\begin{equation}\label{eq.s1}
  \tilde{A}_M=A_0/6 + 33/16, \quad \tilde{B}_M=\frac{D_1^2}{2^{17} (\ln K_0)^2}, \quad \tilde{C}_M=\frac{4\delta \ln K_0}{D_1}, \end{equation}
and for clarity, we recall that\\
\textbullet ${}$ (\S \ref{subsec.lower.bound}) $D_1=\max \{1,2\diam(G\backslash M)\}$,\\ 
\textbullet ${}$ (\S \ref{subsec.lower.bound}) $K_0$ satisfies $K_0\leq \frac{\mu_G(B_{13D_1/2})}{\mu_G(B_{D_1/2})}$, and\\ 
\textbullet ${}$ (\S \ref{subsec.main}) $A_0$ is the doubling constant given in Remark \ref{rk.explicit.C}.

Finally, the expression in Remark \ref{rk.explicit.tits} follows by taking
\begin{equation}\label{eq.explicit.guy}
A_M=\max\left\{\tilde{C}_M,\frac{\tilde{A}_M^2}{\tilde{B}_M}\right\}.
\end{equation}

\begin{remark}\label{rk.last.rk}
The explicit bounds on the probability mentioned in \S \ref{subsub.tits.conseq} for hyperbolic groups and rank-one linear groups are obtained by plugging the upper bounds \eqref{eq.explicit.bound.koubi} and \eqref{eq.explicit.bound.manu} on $\lambda_r$ into
\eqref{eq.the.real.tits.bound} in the proof above. Similarly, for the range of validity of $n \in \N$, one can plug \eqref{eq.explicit.bound.koubi} and \eqref{eq.explicit.bound.manu} in \eqref{eq.lyap.lower} to get an explicit lower bound for $\ell(\mu)$ which then provides an upper bound for the right-hand-side of \eqref{eq.range}.
\end{remark}


\end{document}